\documentclass[12pt,reqno]{smfart}
\usepackage{amsmath,amssymb,smfthm,stmaryrd,epic,enumerate}
\usepackage[frenchb]{babel}
\usepackage[latin1]{inputenc}
\usepackage[all]{xy}
\usepackage{a4wide}

\usepackage{pstricks}

\usepackage[active]{srcltx}

\usepackage{lmodern}

\usepackage{hyperref}

\NumberTheoremsIn{subsection}\SwapTheoremNumbers

\begin{document}
\newtheorem{prop-defi}[smfthm]{Proposition-DÈfinition}
\newtheorem{notas}[smfthm]{Notations}
\newtheorem{nota}[smfthm]{Notation}
\newtheorem{defis}[smfthm]{DÈfinitions}
\newtheorem{hypo}[smfthm]{HypothËse}

\def\Xm{{\mathbb X}}
\def\Um{{\mathbb U}}
\def\Am{{\mathbb A}}
\def\Fm{{\mathbb F}}
\def\Mm{{\mathbb M}}
\def\Nm{{\mathbb N}}
\def\Pm{{\mathbb P}}
\def\Qm{{\mathbb Q}}
\def\Zm{{\mathbb Z}}
\def\Dm{{\mathbb D}}
\def\Cm{{\mathbb C}}
\def\Rm{{\mathbb R}}
\def\Gm{{\mathbb G}}
\def\Lm{{\mathbb L}}
\def\Km{{\mathbb K}}
\def\Om{{\mathbb O}}
\def\Em{{\mathbb E}}

\def\BC{{\mathcal B}}
\def\QC{{\mathcal Q}}
\def\TC{{\mathcal T}}
\def\ZC{{\mathcal Z}}
\def\AC{{\mathcal A}}
\def\CC{{\mathcal C}}
\def\DC{{\mathcal D}}
\def\EC{{\mathcal E}}
\def\FC{{\mathcal F}}
\def\GC{{\mathcal G}}
\def\HC{{\mathcal H}}
\def\IC{{\mathcal I}}
\def\JC{{\mathcal J}}
\def\KC{{\mathcal K}}
\def\LC{{\mathcal L}}
\def\MC{{\mathcal M}}
\def\NC{{\mathcal N}}
\def\OC{{\mathcal O}}
\def\PC{{\mathcal P}}
\def\UC{{\mathcal U}}
\def\VC{{\mathcal V}}
\def\XC{{\mathcal X}}
\def\SC{{\mathcal S}}
\def\RC{{\mathcal R}}

\def\BF{{\mathfrak B}}
\def\AF{{\mathfrak A}}
\def\GF{{\mathfrak G}}
\def\EF{{\mathfrak E}}
\def\CF{{\mathfrak C}}
\def\DF{{\mathfrak D}}
\def\JF{{\mathfrak J}}
\def\LF{{\mathfrak L}}
\def\MF{{\mathfrak M}}
\def\NF{{\mathfrak N}}
\def\XF{{\mathfrak X}}
\def\UF{{\mathfrak U}}
\def\KF{{\mathfrak K}}
\def\FF{{\mathfrak F}}

\def \longmapright#1{\smash{\mathop{\longrightarrow}\limits^{#1}}}
\def \mapright#1{\smash{\mathop{\rightarrow}\limits^{#1}}}
\def \lexp#1#2{\kern \scriptspace \vphantom{#2}^{#1}\kern-\scriptspace#2}
\def \linf#1#2{\kern \scriptspace \vphantom{#2}_{#1}\kern-\scriptspace#2}
\def \linexp#1#2#3 {\kern \scriptspace{#3}_{#1}^{#2} \kern-\scriptspace #3}

\def \coim {\mathop{\mathrm{Coim}}\nolimits}
\def \Ext{\mathop{\mathrm{Ext}}\nolimits}
\def \ad{\mathop{\mathrm{ad}}\nolimits}
\def \sh{\mathop{\mathrm{Sh}}\nolimits}
\def \irr{\mathop{\mathrm{Irr}}\nolimits}
\def \FH{\mathop{\mathrm{FH}}\nolimits}
\def \FPH{\mathop{\mathrm{FPH}}\nolimits}
\def \coh{\mathop{\mathrm{Coh}}\nolimits}
\def \res{\mathop{\mathrm{res}}\nolimits}
\def \op{\mathop{\mathrm{op}}\nolimits}
\def \rec {\mathop{\mathrm{rec}}\nolimits}
\def \art{\mathop{\mathrm{Art}}\nolimits}
\def \hyp {\mathop{\mathrm{Hyp}}\nolimits}
\def \cusp {\mathop{\mathrm{Cusp}}\nolimits}
\def \scusp {\mathop{\mathrm{Scusp}}\nolimits}
\def \Iw {\mathop{\mathrm{Iw}}\nolimits}
\def \JL {\mathop{\mathrm{JL}}\nolimits}
\def \speh {\mathop{\mathrm{Speh}}\nolimits}
\def \isom {\mathop{\mathrm{Isom}}\nolimits}
\def \Vect {\mathop{\mathrm{Vect}}\nolimits}
\def \groth {\mathop{\mathrm{Groth}}\nolimits}
\def \hom {\mathop{\mathrm{Hom}}\nolimits}
\def \deg {\mathop{\mathrm{deg}}\nolimits}
\def \val {\mathop{\mathrm{val}}\nolimits}
\def \det {\mathop{\mathrm{det}}\nolimits}
\def \rep {\mathop{\mathrm{Rep}}\nolimits}
\def \spec {\mathop{\mathrm{Spec}}\nolimits}
\def \fr {\mathop{\mathrm{Fr}}\nolimits}
\def \frob {\mathop{\mathrm{Frob}}\nolimits}
\def \ker {\mathop{\mathrm{Ker}}\nolimits}
\def \im {\mathop{\mathrm{Im}}\nolimits}
\def \Red {\mathop{\mathrm{Red}}\nolimits}
\def \red {\mathop{\mathrm{red}}\nolimits}
\def \aut {\mathop{\mathrm{Aut}}\nolimits}
\def \diag {\mathop{\mathrm{diag}}\nolimits}
\def \spf {\mathop{\mathrm{Spf}}\nolimits}
\def \Def {\mathop{\mathrm{Def}}\nolimits}
\def \twist {\mathop{\mathrm{Twist}}\nolimits}
\def \scusp {\mathop{\mathrm{Scusp}}\nolimits}
\def \Id {{\mathop{\mathrm{Id}}\nolimits}}
\def \lie {{\mathop{\mathrm{Lie}}\nolimits}}
\def \Ind{\mathop{\mathrm{Ind}}\nolimits}
\def \ind {\mathop{\mathrm{ind}}\nolimits}
\def \loc {\mathop{\mathrm{Loc}}\nolimits}
\def \top {\mathop{\mathrm{Top}}\nolimits}
\def \ker {\mathop{\mathrm{Ker}}\nolimits}
\def \coker {\mathop{\mathrm{Coker}}\nolimits}
\def \gal {{\mathop{\mathrm{Gal}}\nolimits}}
\def \Nr {{\mathop{\mathrm{Nr}}\nolimits}}
\def \rn {{\mathop{\mathrm{rn}}\nolimits}}
\def \tr {{\mathop{\mathrm{Tr~}}\nolimits}}
\def \Sp {{\mathop{\mathrm{Sp}}\nolimits}}
\def \st {{\mathop{\mathrm{St}}\nolimits}}
\def \sp{{\mathop{\mathrm{Sp}}\nolimits}}
\def \perv{\mathop{\mathrm{Perv}}\nolimits}
\def \tor {{\mathop{\mathrm{Tor}}\nolimits}}
\def \nrd {{\mathop{\mathrm{Nrd}}\nolimits}}
\def \nilp {{\mathop{\mathrm{Nilp}}\nolimits}}
\def \obj {{\mathop{\mathrm{Obj}}\nolimits}}
\def \cl {{\mathop{\mathrm{cl}}\nolimits}}
\def \gr {{\mathop{\mathrm{gr}}\nolimits}}
\def\HT{{\mathop{\mathcal{HT}}\nolimits}}

\def \rem{{\noindent\textit{Remarque:~}}}
\def \ext {{\mathop{\mathrm{Ext}}\nolimits}}
\def \End {{\mathop{\mathrm{End}}\nolimits}}

\def\semi{\mathrel{>\!\!\!\triangleleft}}
\let \DS=\displaystyle

\def \hi{\HC}

\setcounter{secnumdepth}{3} \setcounter{tocdepth}{3}

\def \Fil{\mathop{\mathrm{Fil}}\nolimits}
\def \CoFil{\mathop{\mathrm{CoFil}}\nolimits}
\def \Fill{\mathop{\mathrm{Fill}}\nolimits}
\def \CoFill{\mathop{\mathrm{CoFill}}\nolimits}
\def\SF{{\mathfrak S}}
\def\PF{{\mathfrak P}}
\def \EFil{\mathop{\mathrm{EFil}}\nolimits}
\def \ECoFil{\mathop{\mathrm{ECoFil}}\nolimits}
\def \EFill{\mathop{\mathrm{EFill}}\nolimits}
\def \FP{\mathop{\mathrm{FP}}\nolimits}

\let \longto=\longrightarrow
\let \oo=\infty

\let \d=\delta
\let \k=\kappa

\newcommand{\marque}{\addtocounter{smfthm}{1}
{\smallskip \noindent \textit{\thesmfthm}~---~}}

\renewcommand\atop[2]{\ensuremath{\genfrac..{0pt}{1}{#1}{#2}}}

\title[Faisceaux pervers entiers d'Harris-Taylor]{Sur les extensions intermÈdiaires des systËmes locaux d'Harris-Taylor}

\alttitle{On intermediate extensions of Harrris-Taylor's local systems}

\author{Boyer Pascal}
\email{boyer@math.univ-paris13.fr}
\address{ArShiFo ANR-10-BLAN-0114}

\frontmatter

\begin{abstract}
Dans le contexte des variÈtÈs de Shimura unitaires simples ÈtudiÈes par Harris et Taylor,
nous avons construit dans \cite{boyer-LT} deux filtrations du faisceau pervers des cycles
Èvanescents. Les graduÈs de la premiËre sont les $p$-extensions intermÈdiaires de certains
systËmes locaux dits d'Harris-Taylor tandis que ceux de la seconde, obtenue par dualitÈ,
sont les $p+$-extensions intermÈdiaires. Dans ce papier nous dÈcrivons
la diffÈrence entre ces $p$ et $p+$ extensions intermÈdiaires. PrÈcisÈment nous
montrons que lorsque le systËme local d'Harris-Taylor est associÈ ‡ une reprÈsentation
irrÈductible cuspidale dont la rÈduction modulo $l$ est supercuspidale, ces extensions
intermÈdiaires sont les mÍmes. Dans le cas o˘ la rÈduction modulo $l$ n'est que
cuspidale nous dÈcrivons la $l$-torsion de leur diffÈrence.

\end{abstract}

\begin{altabstract}
In the geometric situation of some simple unitary Shimura varieties studied by Harris and 
Taylor, we have built in \cite{boyer-LT} two filtrations of the perverse sheaf of vanishing cycles.
The graduate of the first are the $p$-intermediate extension of some local Harris-Taylor's local
systems, while for the second, obtained by duality, they are the $p+$-intermediate extensions.
In this work, we describe the difference between these $p$ and $p+$ intermediate 
extension. Precisely, we show, in the case where the local system is associated to an 
irreducible cuspidal representation whose reduction modulo $l$ is supercuspidal, that
the two intermediate extensions are the same. Otherwise, if the reduction modulo $l$
is just cuspidal, we describe the $l$-torsion of their difference.

\end{altabstract}



\altkeywords{Shimura varieties, formal modules, Langlands correspondences, Jacquet-Langlands
correspondences,
monodromy filtration, weight-monodromy conjecture, perverse sheaves, vanishing cycles}

\maketitle

\pagestyle{headings} \pagenumbering{arabic}

\section*{Introduction}
\renewcommand{\theequation}{\arabic{equation}}
\backmatter

Cet article s'inscrit dans un programme visant ‡ prouver le lemme d'Ihara pour $U(n,1)$. 
Avant d'en dÈcrire les rÈsultats principaux, dÈtaillons la stratÈgie de ce programme dont le lecteur
trouvera plus de dÈtails dans cf. \cite{boyer-ihara}.
\begin{itemize}
\item La premiËre Ètape consiste ‡ traduire l'ÈnoncÈ du lemme d'Ihara en un ÈnoncÈ similaire portant
sur la cohomologie d'une variÈtÈ de Shimura unitaire de type Kottwitz-Harris-Taylor.

\item Pour Ètudier la $\overline \Zm_l$-cohomologie de ces variÈtÈs de Shimura, on utilise la
suite spectrale des cycles Èvanescents en une place $p \neq l$. Le faisceau pervers des cycles
Èvanescents est dÈcoupÈ sur $\overline \Zm_l$ en termes de versions entiËres des
extensions intermÈdiaires des systËmes locaux d'Harris-Taylor. On est alors amenÈ ‡
contrÙler la torsion dans la cohomologie de ces $\overline \Zm_l$-systËmes locaux d'Harris-Taylor. 

\item La torsion de la cohomologie de toute la variÈtÈ de Shimura est plus facilement contrÙlable car on 
peut la calculer en utilisant une place annexe quelconque, en particulier o˘ il y a bonne rÈduction. 
Quitte ‡ localiser en un idÈal maximal bien choisi d'une algËbre de Hecke, on montre dans \cite{boyer-imj} 
que la cohomologie est concentrÈe en degrÈ mÈdian et sans torsion.

\item Dans \cite{boyer-LT} nous avons construit deux filtrations du faisceau pervers des cycles 
Èvanescents, les graduÈs de la premiËre s'identifiant aux $p$-extensions intermÈdiaires des systËmes 
locaux d'Harris-Taylor alors que ceux de la seconde, obtenue par dualitÈ,
sont les $p+$-extensions intermÈdiaires. 

\item En utilisant que les strates de Newton sont affines et que la cohomologie de la variÈtÈ de Shimura,
aprËs localisation, est sans torsion, il n'est pas trop difficile de contrÙler la cohomologie
de ces $p$-extensions intermÈdiaires avant le degrÈ mÈdian. La dualitÈ de Verdier permet alors
d'obtenir que la cohomologie des $p+$-extensions intermÈdiaires est aussi sans torsion
en degrÈ supÈrieur au degrÈ mÈdian, de sorte que, dans le cas o˘ les $p$ et $p+$ extensions 
intermÈdiaires coÔncident, on obtient la nullitÈ de la torsion en tout degrÈ.
\end{itemize}

Le but de ce travail est ainsi de comprendre la diffÈrence entre ces deux versions $p$ et $p+$
des faisceaux pervers d'Harris-Taylor. Le rÈsultat principal donnÈ au thÈorËme 
\ref{theo-principal}, est qu'elles sont Ègales lorsque la reprÈsentation irrÈductible cuspidale
associÈe au systËme local possËde une rÈduction modulo $l$ qui est supercuspidale.
En revanche dans le cas o˘ cette rÈduction modulo $l$ est simplement cuspidale,
ces deux extensions intermÈdiaires sont distinctes et on en dÈcrit la $l$-torsion
de leur quotient au \S \ref{para-cuspi} ‡ l'aide de la description de la rÈduction modulo $l$
des reprÈsentations de Steinberg dÈcrites dans \cite{boyer-repmodl}. Nous verrons dans \cite{boyer-ihara}
que cette description est suffisante pour contrÙler la torsion en tout degrÈ.

En ce qui concerne l'organisation du papier, on commence par des rappels sur la gÈomÈtrie
\S \ref{para-shimura} et sur les systËmes locaux d'Harris-Taylor, \S \ref{para-HT}, puis on
Ènonce le rÈsultat principal, thÈorËme \ref{theo-principal}. Pour simplifier la dÈmonstration,
on commence, proposition \ref{prop-decompo}, par dÈcouper $\Psi_\IC$ selon les
$\overline \Fm_l$-reprÈsentations irrÈductibles supercuspidales d'un $GL_g(F_v)$ pour 
$1 \leq g \leq d$, cf. la dÈfinition \ref{defi-tau-type} et le thÈorËme \ref{theo-jl} de l'appendice.
Au \S \ref{para-cuspi}, nous Ètudions enfin le cas non supercuspidal en dÈcrivant,
proposition \ref{prop-dec-pervers}, la $l$-torsion du quotient 
entre les $p+$ et $p$ versions des faisceaux pervers d'Harris-Taylor. Nous finissons
en dÈcrivant les faisceaux de cohomologie de la rÈduction 
modulo $l$ d'un faisceau pervers d'Harris-Taylor, lesquels sont connus d'aprËs le rÈsultat
principal de \cite{boyer-LT}: en effet dans le cas non supercuspidal, la filtration de
stratification exhaustice de celui-ci, donnÈe par la proposition  \ref{prop-dec-pervers}, 
fournit une suite spectrale reflÈtÈe par les complexes dits d'induction du 
\S \ref{para-complexe}.

\tableofcontents

\mainmatter

\renewcommand{\theequation}{\arabic{section}.\arabic{subsection}.\arabic{smfthm}}

\section{Faisceaux pervers d'Harris-Taylor entiers}

\subsection{Rappels sur quelques variÈtÈs de Shimura unitaires}
\label{para-shimura}

Soit $F=F^+ E$ un corps CM avec $E/\Qm$ quadratique imaginaire, dont on fixe 
un plongement rÈel $\tau:F^+ \hookrightarrow \Rm$. 

\begin{nota} Pour toute place finie $w$ de $F$, on note $F_w$ le complÈtÈ de $F$ en cette place,
$\OC_w$ son anneau des entiers d'idÈal maximal $\PC_w$ et de corps rÈsiduel $\kappa(w)$.
\end{nota}

Soit $B$ une algËbre ‡ 
division centrale sur $F$ de dimension $d^2$ telle qu'en toute place $x$ de $F$,
$B_x$ est soit dÈcomposÈe soit une algËbre ‡ division et on suppose $B$ 
munie d'une involution de
seconde espËce $*$ telle que $*_{|F}$ est la conjugaison complexe $c$. Pour
$\beta \in B^{*=-1}$, on note $\sharp_\beta$ l'involution $x \mapsto x^{\sharp_\beta}=\beta x^*
\beta^{-1}$ et $G/\Qm$ le groupe de similitudes, notÈ $G_\tau$ dans \cite{h-t}, dÈfini
pour toute $\Qm$-algËbre $R$ par 
$$
G(R)  \simeq   \{ (\lambda,g) \in R^\times \times (B^{op} \otimes_\Qm R)^\times  \hbox{ tel que } 
gg^{\sharp_\beta}=\lambda \}
$$
avec $B^{op}=B \otimes_{F,c} F$. 
Dans \cite{h-t}, les auteurs justifient l'existence de $G$ comme ci-dessus tel qu'en outre
\begin{itemize}
\item les signatures de $G(\Rm)$ sont $(1,d-1)$ pour le plongement $\tau$ et
$(0,d)$ pour les autres;

\item pour $p=u \lexp c u$ dÈcomposÈ dans $E$,
$$G(\Qm_p) \simeq (\Qm_p)^\times \times \prod_{i=1}^r (B_{v_i}^{op})^\times$$ 
o˘ $v=v_1,v_2,\cdots,v_r$ sont les places de $F$ au dessus de la place $u$ de $E$.
\end{itemize}

Pour tout sous-groupe compact $U^p$ de $G(\Am^{\oo,p})$ et $m=(m_1,\cdots,m_r) \in \Zm_{\geq 0}^r$, on pose
$$U^p(m)=U^p \times \Zm_p^\times \times \prod_{i=1}^r \ker ( \OC_{B_{v_i}}^\times \longto
(\OC_{B_{v_i}}/\PC_{v_i}^{m_i})^\times )$$

\begin{nota} \label{nota-m1}
On note $\IC$ l'ensemble des sous-groupes compacts ouverts $U^p(m)$ tels qu'il existe une place $x$ pour laquelle la projection de $U^p$ sur $G(\Qm_x)$ ne contienne
aucun ÈlÈment d'ordre fini autre que l'identitÈ, cf. \cite{h-t} bas de la page 90.
Pour $m$ comme ci-dessus, on a une application
$$m_1: \IC \longrightarrow \Nm.$$
\end{nota}

\begin{defi}
Pour tout $I \in \IC$, on note $X_I \rightarrow \spec \OC_v$ \og la variÈtÈ de Shimura
associÈe ‡ $G$\fg{} construite dans \cite{h-t} et $X_\IC=(X_I)_{I \in \IC}$
le schÈma de Hecke relativement au groupe  $G(\Am^\oo)$, au sens de \cite{boyer-invent2}
\end{defi}

\rem les morphismes de restriction du niveau $r_{J,I}:X_J \rightarrow X_I$ sont finis et plats.
et mÍme Ètales quand $m_1(J)=m_1(I)$.

\begin{notas} (cf. \cite{boyer-invent2} \S 1.3) \label{nota-strate}
Pour $I \in \IC$, on note:
\begin{itemize}
\item $X_{I,s}$ la fibre spÈciale de $X_I$ et $X_{I,\bar s}:=X_{I,s} \times \spec \overline \Fm_p$ 
la fibre spÈciale gÈomÈtrique.

\item Pour tout $1 \leq h \leq d$, $X_{I,\bar s}^{\geq h}$ (resp. $X_{I,\bar s}^{=h}$)
dÈsigne la strate fermÈe (resp. ouverte) de Newton de hauteur $h$, i.e. le sous-schÈma dont la
partie connexe du groupe de Barsotti-Tate en chacun de ses points gÈomÈtriques
est de rang $\geq h$ (resp. Ègal ‡ $h$).

\item On notera aussi $X^{\geq 0}_{I,\bar s}:=X_I$.
\end{itemize}
\end{notas}

\rem pour tout $1 \leq h \leq d$, la strate de Newton de hauteur $h$ est
de pure dimension $d-h$; le systËme projectif associÈ dÈfinit alors
un schÈma de Hecke $X_{\IC,\bar s}^{\geq h}$ (resp. $X_{\IC,\bar s}^{=h}$)
pour $\Gm=G(\Am^\oo)$, cf. \cite{h-t} III.4.4, lisse 
dans le cas de bonne rÈduction, i.e. quand $m_1=0$.

%

\begin{notas} Nous utiliserons les notations suivantes:
$$i^h:X^{\geq h}_{\IC,\bar s} \hookrightarrow X^{\geq 1}_{\IC,\bar s}, \quad
j^{\geq h}: X^{=h}_{\IC,\bar s} \hookrightarrow X^{\geq h}_{\IC,\bar s}$$
ainsi que $j^{=h}:=i^h \circ j^{\geq h}$.
\end{notas}

\subsection{SystËmes locaux d'Harris-Taylor}
\label{para-HT}

\begin{nota} Pour tout $h \geq 1$, on note $D_{v,h}$ l'algËbre ‡ division centrale sur $F_v$
d'invariant $1/h$.
\end{nota}

\begin{defi}
Pour $\pi_v$ une reprÈsentation irrÈductible cuspidale de $GL_g(F_v)$ et $t \geq 1$, on note
$\pi_v[t]_D$ la reprÈsentation de $D_{v,tg}^\times$ associÈe ‡ $\st_t(\pi_v)$ par
la correspondance de Jacquet-Langlands. 
\end{defi}

\rem toute reprÈsentation irrÈductible de $D_{v,h}^\times$ est de la forme $\pi_v[t]_D$ 
pour $h=tg$.

\begin{defi} \label{defi-LC}
Dans \cite{h-t}, les auteurs, via les variÈtÈs d'Igusa de premiËre et seconde
espËce, associent ‡ toute reprÈsentation $\rho_v$ de l'ordre maximal $\DC_{v,h}^\times$
de $D_{v,h}^\times$, un systËme local $\LC(\rho_v)$ sur $X^{=h}_{\IC,\bar s}$.
\end{defi}

\begin{nota}
Pour $\pi_v$ une $\overline \Qm_l$-reprÈsentation irrÈductible cuspidale de $GL_g(F_v)$
et $t \geq 1$, on note $\LC(\pi_v,t)$ le $\overline \Qm_l$-systËme local sur $X^{=tg}_{\IC,\bar s}$ 
associÈ ‡ la restriction de $\pi_v[t]_D$ ‡ $\DC_{v,tg}^\times$. 
Si $\Gamma$ est un rÈseau stable de $\pi_v[t]_D$, on notera
$$\LC_\Gamma(\pi_v,t)$$
le $\overline \Zm_l$-systËme local associÈ. 
\end{nota} 

\begin{defi} \label{defi-syst-local-type}
Un $\overline \Zm_l$-systËme local $\LC$ sur $X^{=h}_{\IC,\bar s}$ sera dit de type 
$\varrho$, pour $\varrho$ une $\overline \Fm_l$-reprÈsentation irrÈductible supercuspidale de 
$GL_g(F_v)$ avec $g|h$, si $\LC$ est libre et
$\LC \otimes_{\overline \Zm_l} \overline \Qm_l$ est une somme directe de systËmes locaux 
$\LC(\rho_v)$ o˘ $\rho_v$ est une $\overline \Qm_l$-reprÈsentation de $\DC_{v,h}^\times$
de type $\varrho$ au sens de la dÈfinition \ref{defi-tau-type}.
\end{defi}

Rappelons, cf. \cite{boyer-invent2}, qu'un systËme local d'Harris-Taylor $\LC$ sur
$X^{=h}_{\IC,\bar s}$ est induit, i.e.
$$\LC:= \LC_1 \times_{P_{h,d}(F_v)} GL_d(F_v),$$
o˘ 
\begin{itemize}
\item $X^{=h}_{\IC,\bar s,1}$ est une rÈunion de composantes irrÈductibles de 
$X^{=h}_{\IC,\bar s}$ munie d'une action du parabolique standard $P_{h,d}(F_v)$ de Levi
$GL_h(F_v) \times GL_{d-h}(F_v)$,

\item $\LC_1$ est la restriction de $\LC$ ‡ la strate $X^{=h}_{\IC,\bar s,1}$.
\end{itemize}
Le systËme local $\LC_1$ est muni, cf. \cite{boyer-invent2} \S 1.4.2, d'une action de 
$G(\Am^{\oo,p}) \times P_{h,d}(F_v) \times \Zm$ tel que le sous-groupe unipotent
de $P_{h,d}(F_v)$ agit trivialement alors que l'action du facteur $GL_h(F_v)$ de son Levi
agit via $\val\circ \det: GL_h(F_v) \twoheadrightarrow \Zm$. 

\begin{nota} \label{nota-HT} 
Pour $\Pi_t$ une reprÈsentation de $GL_h(F_v)$, on introduit alors 
$$HT(\pi_v,\Pi_t)(n):=\LC(\pi_v,t)[d-tg] \otimes \Pi_t \otimes \Xi^{\frac{tg-d+n}{2}} \otimes \Lm(\pi_v )$$
o˘ $\Lm^\vee$ est la correspondance Langlands sur $F_v$, 
$$\Xi:\frac{1}{2} \Zm \longrightarrow \overline \Zm_l^\times$$ 
est dÈfinie par $\Xi(\frac{1}{2})=q^{1/2}$ et
\begin{itemize}
\item $GL_h(F_v)$ agit diagonalement sur $\Pi_t$ et sur $\LC(\pi_v,t) \otimes \Xi^{\frac{tg-d+n}{2}}$ 
via son quotient $GL_h(F_v) \twoheadrightarrow \Zm$,

\item le groupe de Weil $W_v$ en $v$ agit diagonalement sur $\Lm(\pi_v)$ et le facteur
$\Xi^{\frac{tg-d+n}{2}}$ via l'application $\deg: W_v \twoheadrightarrow \Zm$ qui envoie les
frobenius gÈomÈtriques sur $1$.
\end{itemize}
Une $\overline \Zm_l$-version entiËre sera notÈe $HT_\Gamma(\pi_v,\Pi_t)(n)$ o˘ $\Gamma$
dÈsigne un rÈseau stable, par forcÈment sous la forme d'un produit tensoriel.
\end{nota}

\rem on rappelle que $\pi'_v$ est inertiellement Èquivalente ‡ $\pi_v$ si et seulement
s'il existe un caractËre $\zeta: \Zm \longrightarrow  \overline \Qm_l^\times$ tel que 
$\pi'_v \simeq \pi_v \otimes (\zeta \circ \val \circ \det)$.
Les faisceaux pervers $HT(\pi_v,\Pi_t)(n)$ ne dÈpendent que de la classe d'Èquivalence inertielle de 
$\pi_v$ et sont de la forme 
$$HT(\pi_v,\Pi_t)(n)=e_{\pi_v} \HT(\pi_v,\Pi_t)(n)$$ 
o˘ $\HT(\pi_v,\Pi_t)(n)$ est un faisceau pervers irrÈductible.

\subsection{\'EnoncÈ dans le cas supercuspidal}
\label{para-enonce}

Rappelons que pour $X$ un $\Fm_p$-schÈma et 
$\Lambda=\overline \Qm_l,\overline \Zm_l,\overline \Fm_l$, la $t$-structure usuelle sur 
la catÈgorie dÈrivÈe $D^b_c(X,\Lambda)$ est dÈfinie par:
$$\begin{array}{l}
A \in \lexp p D^{\leq 0}(X,\Lambda)
\Leftrightarrow \forall x \in X,~\hi^k i_x^* A=0,~\forall k >- \dim \overline{\{ x \} } \\
A \in \lexp p D^{\geq 0}(X,\Lambda) \Leftrightarrow \forall x \in X,~\hi^k i_x^! A=0,~\forall k <- 
\dim \overline{\{ x \} }
\end{array}$$
o˘ $i_x:\spec \kappa(x) \hookrightarrow X$ et $\hi^k$ dÈsigne le $k$-iËme faisceau de cohomologie.
On note alors $\lexp p \CC(X,\Lambda)$ le c{\oe}ur de cette $t$-structure: c'est une catÈgorie abÈlienne noethÈrienne et $\Lambda$-linÈaire.

\begin{nota} Les foncteurs cohomologiques associÈs ‡ la $t$-structure perverse ci-avant seront notÈs 
$\lexp p \hi^i$.
\end{nota}

Pour $\Lambda$ un corps, cette $t$-structure est autoduale pour la dualitÈ de Verdier.
Pour $\Lambda=\overline \Zm_l$, on peut munir la catÈgorie abÈlienne $\overline \Zm_l$-linÈaire 
$\lexp p \CC(X,\Lambda)$ d'une thÈorie de torsion $(\TC,\FC)$ o˘ 
$\TC$ (resp. $\FC$) est la sous-catÈgorie
pleine des objets de torsion $T$ (resp. libres $F$) , i.e. tels que $l^N 1_T$ est nul pour $N$ assez grand (resp. $l.1_F$ est un monomorphisme).

\begin{defi} Soit
$$\begin{array}{l}
\lexp {p+} \DC^{\leq 0}(X,\overline \Zm_l):= \{ A \in \lexp p \DC^{\leq 1}(X,\overline \Zm_l):~
\lexp p \hi^1(A) \in \TC \} \\
\lexp {p+} \DC^{\geq 0}(X,\overline \Zm_l):= \{ A \in  \lexp p \DC^{\geq 0}(X,\overline \Zm_l):~
\lexp p \hi^0(A) \in \FC \} \\
\end{array}$$
la $t$-structure duale de c\oe ur $\lexp {p+} \CC(X,\overline \Zm_l)$  muni de sa thÈorie de torsion 
$(\FC,\TC[-1])$ \og duale \fg{} de celle de $\lexp p \CC(X,\overline \Zm_l)$.
\end{defi}

\rem d'aprËs \cite{boyer-torsion} \S 1.3, la sous-catÈgorie pleine $\FC$ de $\lexp p \CC(X,\Lambda)$
est quasi-abÈlienne, i.e. elle admet des noyaux, images, conoyaux et coimages mais la flËche
naturelle 
$$\coim_\FC f \longrightarrow \im_\FC f$$ 
de la coimage vers l'image, n'est pas nÈcessairement
un isomorphisme: si c'est le cas on dit que le morphisme $f$ est \emph{strict}.
Pour $j:U \hookrightarrow X$ une immersion ouverte, on dispose en outre de deux notions d'extensions intermÈdiaires
$$\lexp p j_{!*} \quad \hbox{ et } \quad \lexp {p+} j_{!*}.$$
Le rÈsultat principal que nous allons montrer est donnÈ par le thÈorËme suivant.

\begin{theo} \phantomsection \label{theo-principal}
Soit $\pi_v$ une reprÈsentation irrÈductible cuspidale de $GL_g(F_v)$ telle que
sa rÈduction modulo $l$ est supercuspidale alors pour tout $1 \leq t \leq s=\lfloor \frac{d}{g}
\rfloor$, et pour tout rÈseau stable $\Gamma$ de $\pi_v[t]_D$, on a
$$\lexp p j^{\geq tg}_{!*} \LC_{\Gamma}(\pi_v,t)[d-tg] \simeq \lexp {p+}
j^{\geq tg}_{!*} \LC_{\Gamma}(\pi_v,t)[d-tg].$$
\end{theo}

\rem comme la rÈduction modulo
$l$ de $\pi_v[t]_D$ est irrÈductible, l'indÈpendance relativement au rÈseau $\Gamma$
considÈrÈ est immÈdiate.  Cependant la preuve de ce rÈsultat passe par un ÈnoncÈ similaire
portant sur les systËmes locaux $HT_\Gamma(\pi_v,\Pi_t)(n)$ o˘ le rÈseau est donnÈ par
le faisceau pervers des cycles Èvanescents. PrÈcisÈment nous prouverons la proposition suivante.

\begin{prop} \label{prop-principal}
Soit $\pi_v$ une reprÈsentation irrÈductible cuspidale de $GL_g(F_v)$ telle que
sa rÈduction modulo $l$ est supercuspidale. Alors pour tout $1 \leq t \leq s=\lfloor \frac{d}{g} \rfloor$
il existe une reprÈsentation $\Pi_t$ de $GL_{tg}(F_v)$, un entier $n$ ainsi qu'un rÈseau stable $\Gamma$
de $\HT(\pi_v,\Pi_t)(n)$ tel que 
$$\lexp p j^{=tg}_{!*} \HT_\Gamma(\pi_v,\Pi_t)(n) \simeq \lexp {p+} j^{=tg}_{!*} 
\HT_\Gamma(\pi_v,\Pi_t)(n).$$
\end{prop}

Le thÈorËme dÈcoule alors du lemme suivant en considÈrant le rÈseau produit tensoriel, i.e.
la propriÈtÈ de la proposition prÈcÈdente, ne dÈpend pas de $\Pi_t$ ou 
de $n$ mais seulement d'une structure entiËre quelconque du systËme local $\LC(\pi_v,t)$.

\begin{lemm} \label{lem-principal} S'il existe un rÈseau stable $\Gamma$ de
$\HT(\pi_v,\Pi_t)(n)$ tel que la proposition prÈcÈdente est valable, alors le rÈsultat est valable
pour tout rÈseau stable $\Gamma'$ de $\HT(\pi_v,\Pi_t)(n)$.
\end{lemm}

\begin{proof}
On raisonne par rÈcurrence sur $t$ de $d$ ‡ $s$. Le cas $t=s$ dÈcoule des ÈgalitÈs
$$\lexp p j^{=sg}_{!*} \HT_{\Gamma'}(\pi_v,\Pi_t)(n)= j^{=sg}_! \HT_{\Gamma'}(\pi_v,\Pi_t)(n)
= j^{=sg}_* \HT_{\Gamma'}(\pi_v,\Pi_t)(n) = \lexp {p+} j^{=sg}_{!*} \HT_{\Gamma'}(\pi_v,\Pi_t)(n).$$
Supposons alors le rÈsultat acquis jusqu'au rang $t+1$.
D'aprËs \cite{juteau}, on a les triangles distinguÈs
$$\lexp p j^{=tg}_{!*} \HT_\Gamma(\pi_v,\Pi_t)(n) \longrightarrow 
\lexp {p+} j^{=tg}_{!*} \HT_\Gamma(\pi_v,\Pi_t)(n) \longrightarrow
i^{tg}_* \lexp p \hi^0_{tor} i^{tg,*} j^{=tg}_* \HT_\Gamma(\pi_v,\Pi_t)(n) [1] \leadsto$$
ainsi que pour le foncteur $\Fm:=- \otimes^{\Lm}_{\overline \Zm_l} \overline \Fm_l$
$$\Fm \lexp p j^{=tg}_{!*} \HT_\Gamma(\pi_v,\Pi_t)(n) \longrightarrow j^{=tg}_{!*} 
\Fm \HT_\Gamma(\pi_v,\Pi_t)(n) \longrightarrow \lexp p \hi^0 \Fm \bigl ( 
i^{tg}_* \lexp p \hi^0_{tor} i^{tg,*} j^{=tg}_* \HT_\Gamma(\pi_v,\Pi_t)(n) \bigr ) \leadsto$$
de sorte que l'ÈgalitÈ des extensions intermÈdiaires pour $\Gamma'$ revient ‡ demander que
$\Fm \lexp p j^{=tg}_{!*} \HT_{\Gamma'}(\pi_v,\Pi_t)(n) \simeq j^{=tg}_{!*}  \Fm \HT_{\Gamma'}
(\pi_v,\Pi_t)(n)$.
Rappelons la suite exacte courte
$$0 \rightarrow i^{tg}_* \lexp p \hi^{-1} i^{tg,*} j^{=tg}_* \Fm  \HT_\Gamma(\pi_v,\Pi_t)(n) 
\longrightarrow j^{=tg}_! \Fm \HT_\Gamma(\pi_v,\Pi_t)(n) \longrightarrow
j^{=tg}_{!*} \Fm  \HT_\Gamma(\pi_v,\Pi_t)(n)  \rightarrow 0.$$
Comme $j^{\geq tg}$ est affine, les foncteurs $j^{=tg}_!$, $j^{=tg}_*$ sont exacts
tout comme $i^{tg,*}$ et $i'^{tg}_*$. Ainsi le foncteur $\Fm$ commute avec ceux-ci et comme la torsion
de $\lexp p \hi^0 i^{tg,*}j^{=tg}_* \HT_\Gamma(\pi_v,\Pi_t)(n)$ est nulle alors
$$ i^{tg}_* \lexp p \hi^{-1} i^{tg,*} j^{=tg}_* \Fm  \HT_\Gamma(\pi_v,\Pi_t)(n) \simeq 
\Fm \bigl (  i^{tg}_* \lexp p \hi^{-1} i^{tg,*} j^{=tg}_* \HT_\Gamma(\pi_v,\Pi_t)(n) \bigr ).$$
Comme $ \lexp p \hi^{-1} i^{tg,*} j^{=tg}_* = \lexp p \hi^{-1} i^{tg,*} j^{=tg}_{!*}$ et que
$j^{=tg}_{!*} \Fm \HT_\Gamma(\pi_v,\Pi_t)(n)$ est semi-simple en tant que faisceau, 
on en dÈduit que, dans le groupe de Grothendieck,
 $\Bigl [ i^{tg}_* \lexp p \hi^{-1} i^{tg,*} j^{=tg}_* \Fm  \HT_{\Gamma}(\pi_v,\Pi_t)(n) \Bigr ] $
ne dÈpend pas du rÈseau $\Gamma$. Du diagramme
$$\xymatrix{
\Fm \bigl ( i^{tg}_* \lexp p \hi^{-1} i^{tg,*} j^{=tg}_* \HT_{\Gamma'}(\pi_v,\Pi_t)(n) \bigr ) \ar@^{^{(}->}[r] & 
\Fm \bigl ( j^{=tg}_! \HT_{\Gamma'}(\pi_v,\Pi_t)(n) \bigr ) \ar@{->>}[r] &
\Fm \bigl ( j^{=tg}_{!*} \HT_{\Gamma'} (\pi_v,\Pi_t)(n) \bigr ) \\
 i^{tg}_* \lexp p \hi^{-1} i^{tg,*} j^{=tg}_*  \Fm \HT_{\Gamma'}(\pi_v,\Pi_t)(n) \ar@^{^{(}->}[r] & 
j^{=tg}_! \Fm \HT_{\Gamma'}(\pi_v,\Pi_t)(n) \bigr ) \ar@{->>}[r] \ar@{=}[u] &
j^{=tg}_{!*} \Fm \HT_{\Gamma'} (\pi_v,\Pi_t)(n) \bigr )
}$$
on remarque qu'il suffit alors de montrer que 
$\Bigl [ \Fm \bigl (  i^{tg}_* \lexp p \hi^{-1} i^{tg,*} j^{=tg}_* \HT_{\Gamma'}(\pi_v,\Pi_t)(n) \bigr ) \Bigr ]$
ne dÈpend pas de $\Gamma'$. 

D'aprËs \cite{boyer-invent2} le faisceau pervers libre
$ i^{tg}_* \lexp p \hi^{-1} i^{tg,*} j^{=tg}_* \HT_{\Gamma'}(\pi_v,\Pi_t)(n)$ est, sur $\overline \Qm_l$, 
extension de faisceaux pervers de la forme $j^{=t'g}_{!*} \HT(\pi_v,\Pi_{t'})(n')$, lesquels 
sur $\overline \Zm_l$, d'aprËs l'hypothËse de rÈcurrence, ne possËdent qu'une notion d'extension
intermÈdiaire. Ainsi donc
leur image par $\Fm$ est $j^{=t'g}_{!*} \Fm \HT(\pi_v,\Pi_{t'})(n')$ qui ne dÈpend pas, dans le groupe de
Grothendieck, du rÈseau stable associÈ. MoralitÈ l'image dans le groupe de Grothendieck de 
$\Fm \bigl (  i^{tg}_* \lexp p \hi^{-1} i^{tg,*} j^{=tg}_* \HT_{\Gamma'}(\pi_v,\Pi_t)(n) \bigr )$
est indÈpendante de $\Gamma'$, d'o˘ le rÈsultat.

\end{proof}

\section{Sur le faisceau pervers des cycles proches}

\subsection{Rappels}
\label{para-rappel}

Pour $\Lambda=\overline \Qm_l,\overline \Zm_l,\overline \Fm_l$ et pour tout $I \in \IC$, 
les faisceaux pervers des cycles Èvanescents
$R\Psi_{\eta_v,I}(\Lambda)[d-1](\frac{d-1}{2})$ sur $X_{I,\bar s}$
dÈfinissent un $W_v$-faisceau pervers de Hecke, au sens de la dÈfinition 1.3.6
de \cite{boyer-invent2}, que l'on note $\Psi_{\IC,\Lambda}$.

\rem dans le cas o˘ $\Lambda=\overline \Zm_l$, on notera simplement $\Psi_\IC$.

Rappelons, cf. \cite{boyer-invent2} \S 2.4, que la restriction 
$\Bigl ( \Psi_{\IC,\Lambda} \Bigr )_{|X^{=h}_{\IC,\bar s}}$ du faisceau pervers des
cycles proches ‡ la strate $X^{=h}_{\IC,\bar s}$, est munie d'une action de 
$(D_{v,h}^\times)^0:=\ker \Bigl ( \val \circ \rn: D_{v,h}^\times \longrightarrow \Zm \Bigr )$
et de $\varpi_v^\Zm$ que l'on voit plongÈ dans $F_v^\times \subset D_{v,h}^\times$.

\begin{prop} \label{prop-fbartau} 
(cf. \cite{h-t} proposition IV.2.2 et le \S 2.4 de \cite{boyer-invent2}) \\
On a un isomorphisme $G(\Am^{\oo,v}) \times P_{h,d-h}(F_v) \times W_v$-Èquivariant\footnote{Noter 
le dÈcalage $[d-1]$ dans la dÈfinition de $\Psi_{\IC,\overline \Zm_l}$.} 
$$\ind_{(D_{v,h}^\times)^0 \varpi_v^\Zm}^{D_{v,h}^\times} 
\Bigl ( \hi^{h-d-i} \Psi_{\IC,\overline \Zm_l} \Bigr )_{|X^{=h}_{\IC,\bar s}} \simeq
\bigoplus_{\bar \tau \in \RC_{\overline \Fm_l}(h)} \LC_{\overline \Zm_l}(\UC_{\bar \tau,\Nm}^{h-1-i})$$
o˘ $\LC_{\overline \Zm_l}(\UC^{h-1}_{\bar \tau,\Nm})$ est le systËme local de la dÈfinition \ref{defi-LC}
associÈ ‡ la $D_{v,h}^\times$-reprÈsentation\footnote{La correspondance entre le systËme indexÈ par 
$\IC$ et $\Nm$ est donnÈe par l'application $m_1$ de \ref{nota-m1}.} 
$\UC^\bullet_{\bar \tau,\Nm}={\displaystyle \lim_{\rightarrow} ~\UC^\bullet_{\bar \tau,n}}$
o˘ $\UC_{\bar \tau,n}^\bullet$ est le $\bar \tau$-facteur isotypique de la $D^\times_{v,h}$-reprÈsentation
admissible $\UC^\bullet_n:=H^\bullet (\MC_{LT,n}^{h/F_v},\overline \Zm_l)$ obtenue comme 
la cohomologie de la fibre gÈnÈrique gÈomÈtrique
$$\MC_{LT,n}^{h/F_v}:=\MC_{LT,h,n} \hat \otimes_{\hat F_v^{nr}} \hat{\overline F_v}$$
du schÈma formel de Lubin-Tate reprÈsentant les classes d'isomorphismes des dÈformations par
quasi-isogÈnies du $\OC_v$-module formel de hauteur $h$ et de dimension $1$.
\end{prop}

\begin{nota} \label{nota-L5}
Pour $\bar \tau \in \RC_{ \overline \Fm_l}(h)$,
on notera $\LC_{\overline \Zm_l}(\bar \tau)$ pour $\LC_{\overline \Zm_l} (\UC^{h-1}_{\bar \tau,\Nm})$.
\end{nota}

\begin{defi} Soit $\varrho$ une $\overline \Fm_l$-reprÈsentation irrÈductible supercuspidale
de $GL_g(F_v)$. On note alors $\loc(\varrho)$ le plus petit ensemble de systËmes locaux
sur les strates de Newton ouvertes $X^{=tg}_{\IC,\bar s}$ pour $1 \leq tg \leq d$ tel que
\begin{itemize}
\item pour tout $\bar \tau \in \RC_{ \overline \Fm_l}(h,\varrho)$,
$\loc(\varrho)$ contient les $\LC_{ \overline \Zm_l}(\bar \tau)$;

\item il est stable par le processus suivant: pour $\LC \in \loc(\varrho)$ un systËme 
local sur $X^{=h}_{\IC,\bar s}$ et un Èpimorphisme
strict $j^{= h}_! \LC[d-h] \twoheadrightarrow F$ de noyau $P_F$, les systËmes locaux de la filtration de 
stratification exhaustive de $P_F$ appartiennent ‡ $\loc(\varrho)$. 

\item si $\LC \in \loc(\varrho)$ alors tout rÈseau stable de 
$\LC \otimes_{ \overline \Zm_l} \overline \Qm_l$ appartient aussi ‡ $\loc(\varrho)$.
\end{itemize}
\end{defi}

\rem en particulier $\loc(\varrho)$ contient tous les systËmes locaux d'Harris-Taylor entiers
$\LC_{\Gamma}(\pi_v,t)$ o˘, cf. \ref{defi-tautype}, $\pi_v \in \scusp_i(\varrho)$ avec
$-1 \leq i \leq s(\varrho)=\frac{d}{g(\varrho)}$.

\begin{prop} \label{prop-hyp1} ( cf. \cite{boyer-LT} proposition 2.3.7) \\ 
Soit $\LC \in \loc(\varrho)$ ‡ support dans $X^{=h}_{\IC,\bar s}$ et soit 
$P_\LC$ le noyau de $j^{= h}_{!} \LC[d-h] \twoheadrightarrow \lexp p j^{=h}_{!*} \LC[d-h]$. On note $h'$
minimal tel que $j^{=h',*} P_\LC$ est non nul alors le morphisme d'adjonction
$$j^{=h'}_! j^{= h',*} P_{\LC}  \longrightarrow P_{\LC}$$
est surjectif dans $\lexp p \CC$.
\end{prop}

On note $X^{1 \leq h}:=X^{\geq 1}-X^{\geq h+1}$ et
$j^{1 \leq h}:X^{1†\leq h} \hookrightarrow X^{\geq 1}$. On dÈfinit alors
$$0=\Fil^0_!(\Psi_{\IC}) \subset \Fil^1_!(\Psi_{\IC}) \subset \cdots \subset
\Fil^d_{!}(\Psi_{\IC})=\Psi_{\IC}$$
la filtration de stratification de $\Psi_{\IC}$ au sens de \cite{boyer-torsion}, o˘ pour $L$
est un faisceau pervers libre on note, pour tout $1 \leq h \leq d$, 
$$\Fil^h_{!}(L):=\im_\FC \Bigl ( \lexp {p+} j^{1 \leq h}_! j^{1 \leq h,*} L \longrightarrow L \Bigr ),$$
autrement dit $\Fil^{h}_!(L)/\Fil^{h-1}_!(L)$
est l'image dans $\FC$ du morphisme d'adjonction
$$j^{=h}_! j^{= h,*} \Bigl ( L/\Fil^{h-1}_!L) \Bigr ) \longrightarrow  L/\Fil^{h-1}_!(L)$$
Le rÈsultat principal de \cite{boyer-LT}, rappelÈ dans la proposition suivante, 
est que, pour $L=\Psi_\IC$, ces images dans $\FC$ sont Ègales aux mÍmes coimages.

\begin{prop} \label{prop-hyp2} (cf. \cite{boyer-LT} proposition 2.4.5) \\
La filtration de stratification de $\Psi_\IC$ est saturÈe, i.e. pour tout $1 \leq h \leq d$, 
le conoyau du morphisme d'adjonction
$$j^{= h}_!  j^{= h,*} \Bigl ( \Psi_\IC/\Fil^{h-1}_!(\Psi_{\IC}) \Bigr ) 
\longrightarrow \Fil^h_!(\Psi_{\IC,\bar \tau}) /\Fil^{h-1}_!(\Psi_{\IC,\bar \tau})$$ 
est libre.
\end{prop}

\rem selon \cite{boyer-torsion}, on peut considÈrer la cofiltration de stratification
$$\Psi_\IC = \CoFil_{*,d}(\Psi_\IC) \twoheadrightarrow \CoFil_{*,d-1}(\Psi_\IC) \twoheadrightarrow \cdots 
\twoheadrightarrow \CoFil_{*,1}(\Psi_\IC) \twoheadrightarrow \CoFil_{*,0}(\Psi_\IC)=0$$
o˘, cf. \cite{boyer-torsion} proposition 2.2.5, 
pour tout faisceau pervers libre $L$ et $1 \leq h \leq d$, on note
$$\CoFil_{*,h}(L)=\coim_\FC \Bigl ( L \longrightarrow \lexp {p} j^{1 \leq h}_* j^{1 \leq h,*} L \Bigr ),$$
autrement dit
$\ker \Bigl (\CoFil_{*,h}(L) \twoheadrightarrow \CoFil_{*,h-1} (L) \Bigr )$ est le noyau dans $\FC$ de
$$\ker \bigl (L \twoheadrightarrow \CoFil_{*,h-1} (L) \bigr )
\longrightarrow j^{= h}_* j^{= h,*} \Bigl ( \ker \bigl (L \twoheadrightarrow \CoFil_{!,h-1} (L) 
\bigr ) \Bigr ).$$
Le dual de la proposition prÈcÈdente est que, pour $L=\Psi_\IC$, ce noyau dans
$\lexp p \CC$ est aussi un noyau dans  $\lexp {p+} \CC$.

\subsection{DÈcomposition supercuspidale}

D'aprËs le thÈorËme \ref{theo-jl}, toute $\overline \Fm_l$-reprÈsentation irrÈductible de $D_{v,d}^\times$
est associÈe ‡ une  $\overline \Fm_l$-reprÈsentation irrÈductible supercuspidale 
$\varrho$ de $GL_{g}(F_{v})$ pour $g$ un diviseur de $d=sg$. 
Ainsi la dÈcomposition de la proposition \ref{prop-fbartau} se raffine en une dÈcomposition, cf. la formule
\ref{eq-tau-type}
$$\ind_{(D_{v,h}^\times)^0 \varpi_v^\Zm}^{D_{v,h}^\times} 
\Bigl ( \hi^{h-d-i} \Psi_{\IC,\overline \Zm_l} \Bigr )_{|X^{=h}_{\IC,\bar s}} \simeq
\bigoplus_{g|h} ~\bigoplus_{\varrho \in \scusp_{F_v}(g)} ~ 
\bigoplus_{\bar \tau \in \RC_{\overline \Fm_l}(h,\varrho)}
\LC_{\overline \Zm_l} (\UC_{\bar \tau,\Nm}^{h-1-i}).$$
Le but de ce paragraphe est de montrer le rÈsultat suivant.

\begin{prop} \label{prop-decompo}
Il existe une dÈcomposition
$$\Psi_{\IC} \simeq \bigoplus_{1 \leq g \leq d}\bigoplus_{\varrho \in \scusp_{F_v}(g)} 
\Psi_{\IC,\varrho}$$
o˘ pour tout $\varrho \in \scusp_{F_v}(g)$, le facteur direct $\Psi_{\IC,\varrho}$ est de type $\varrho$
au sens de la dÈfinition \ref{defi-syst-local-type}.
\end{prop}

\begin{proof}
On reprend la filtration de stratification de $\Psi_\IC$
$$0=\Fil^0_!(\Psi_{\IC}) \subset \Fil^1_!(\Psi_{\IC}) \subset \cdots \subset
\Fil^d_{!}(\Psi_{\IC})=\Psi_{\IC}$$
et on raisonne par rÈcurrence sur $r$ de $0$ ‡ $d$ en supposant qu'une telle dÈcomposition
existe pour $\Fil^r_!(\Psi_{\IC})$. Le cas de $r=0$ Ètant clair, supposons le rÈsultat acquis
pour $r-1$ et montrons le pour $r$. On note $\gr^r_!(\Psi_\IC)$ le quotient
$\Fil^r_!(\Psi_{\IC}) / \Fil^{r-1}_!(\Psi_{\IC})$ dont on rappelle que d'aprËs \cite{boyer-LT}
il est isomorphe ‡ $\lexp p \hi^0 i^{r,*} \Psi_\IC$. D'aprËs la proposition \ref{prop-hyp2}, le
morphisme d'adjonction
$$j^{= r}_! j^{= r,*} \gr^r_!(\Psi_\IC) \longrightarrow \gr^r_!(\Psi_\IC)$$
est surjectif dans $\lexp p \CC$ avec
$$j^{= r,*} \gr^r_!(\Psi_\IC) \simeq  \bigoplus_{g|r}\bigoplus_{\varrho \in \scusp_{F_v}(g)} \LC_{r,\varrho},$$
o˘ $\LC_{r,\varrho}$ est un systËme local sur la strate $X^{=r}_{\IC,\bar s}$ de type
$\varrho$ au sens de \ref{defi-syst-local-type}. On peut ainsi Ècrire
$$\gr^r_!(\Psi_\IC) \simeq \bigoplus_{g|r}\bigoplus_{\varrho \in \scusp_{F_v}(g)} 
\gr^r_{!,\varrho} (\Psi_\IC)$$
avec $j^{= r}_! \LC_{r,\varrho}[d-r] \twoheadrightarrow \gr^r_{!,\varrho} (\Psi_\IC)$
dans $\lexp p \CC$. En outre d'aprËs la proposition \ref{prop-hyp1}, 
$\gr^r_{!,\varrho} (\Psi_\IC)$ admet une filtration dont les graduÈs sont des
$\lexp p j^{= r'}_{!*}$ extensions intermÈdiaires de certains systËmes locaux de type
$\varrho$ au sens de \ref{nota-rhoi}, i.e. $\gr^r_{!,\varrho} (\Psi_\IC)$  est un faisceau pervers de type
$\varrho$.

\begin{lemm} Soient 
\begin{itemize}
\item des $\overline \Fm_l$-reprÈsentations irrÈductibles supercuspidales non isomorphes $\varrho$
et $\varrho'$,

\item $\bar \tau$ et $\bar \tau'$ des $\overline \Fm_l$-reprÈsentations
irrÈductibles respectivement de type $\varrho$ et $\varrho'$ au sens de la dÈfinition 
\ref{defi-tau-type} et 

\item $P$ un $p$-faisceau pervers sans torsion tel que
$$0 \rightarrow A' \longrightarrow P \longrightarrow A \rightarrow 0$$
\end{itemize}
o˘:
\begin{itemize}
\item il existe $h$ et $h'$ ainsi que deux systËmes locaux $\LC$ et $\LC'$ sur respectivement
$X^{=h}_{\IC,\bar s}$ et $X^{=h'}_{\IC,\bar s}$, respectivement de type $\varrho$ et 
$\varrho'$, tels que

\item $A$ (resp. $A'$) est isomorphe ‡ $\lexp p j^{= h}_{!*} \LC[d-h]$;
(resp. $\lexp p j^{= h'}_{!*} \LC'[d-h']$).
\end{itemize}
Alors $P \simeq A \oplus A'$.
\end{lemm}

\begin{proof}
Le cas $h=h'$ dÈcoule de la proposition \ref{prop-scindage}. Supposons pour commencer
que $h>h'$ de sorte que $X^{\geq h}_{\IC,\bar s} \subset X^{\geq h'}_{\IC,\bar s}$ et 
traitons tout d'abord le cas o˘ les coefficients sont $\overline \Qm_l$. La flËche d'adjonction
$$P \longrightarrow j^{= h'}_* j^{= h',*} P$$
a, d'aprËs \cite{boyer-invent2}, pour image $A'$ d'o˘ le rÈsultat. Revenons aux coefficients
$\overline \Zm_l$ et considÈrons la $t$-structure $\tilde p$ obtenue en recollant
\begin{itemize}
\item la $t$ structure usuelle $p$, cf. le dÈbut du \S \ref{para-enonce}, sur
l'ouvert $X^{\geq h'}_{\IC,\bar s}-X^{\geq h}_{\IC,\bar s}$ avec 

\item la $t$-structure $p[-1]$
sur $X^{\geq h}_{\IC,\bar s}$ o˘ $p$ est encore la $t$-structure usuelle. 
\end{itemize}
Notons alors que $A[-1]$ est $\tilde p$-pervers ainsi que $A'$ puisque
$$\lexp p j^{\geq h'}_{!*} \LC'[d-h'] = \lexp {\tilde p} j^{h' \leq h}_{!} \Bigl ( 
\lexp p j^{h' \leadsto h}_{!*} \LC'[d-h'] \Bigr )$$
o˘ 
$$j^{h' \leq h}: X^{\geq h'}_{\IC,\bar s}-X^{\geq h}_{\IC,\bar s} \hookrightarrow
X^{\geq h'}_{\IC,\bar s} 
\quad \hbox{ et } \quad 
j^{h' \leadsto h}:X^{=h'}_{\IC,\bar s} \hookrightarrow 
X^{\geq h'}_{\IC,\bar s} - X^{\geq h}_{\IC,\bar s}.$$
Ainsi $P$ dÈfinit une flËche $A[-1] \longrightarrow A'$ de $\tilde p$-faisceaux pervers
dont l'image est, d'aprËs le cas de $\overline \Qm_l$, 
contenue dans la torsion de $i^{h'}_*\lexp {\tilde p} j^{h' \leq h}_{!} \Bigl ( 
\lexp p j^{h' \leadsto h}_{!*} \LC'[d-h'] \Bigr )$ laquelle est donc ‡ support dans
$X^{\geq h+1}_{\IC,\bar s}$. On conclut en notant que toute flËche de 
$\lexp p j^{\geq h}_{!*} \LC[d-h]$
dans un faisceau pervers ‡ support dans $X^{\geq h+1}_{\IC,\bar s}$ est nulle.

ConsidÈrons ‡ prÈsent le cas $h<h'$. Pour les coefficients $\overline \Qm_l$, l'image
du morphisme d'adjonction
$$j^{= h}_! j^{= h,*} P \longrightarrow P$$
est, d'aprËs \cite{boyer-invent2}, isomorphe ‡ $A$ d'o˘ le rÈsultat. Pour les coefficients
$\overline \Zm_l$, considÈrons la $t$-structure $\tilde p$ obtenue en recollant
celle usuelle $p$ sur $X^{\geq h}_{\IC,\bar s}-X^{\geq h'}_{\IC,\bar s}$ avec la 
$t$-structure $p[1]$ sur $X^{\geq h'}_{\IC,\bar s}$ o˘ $p$ est encore la $t$-structure usuelle.
L'extension $P$ dÈfinit alors une flËche entre $\tilde p$ faisceaux pervers
$A \longrightarrow A'[1]$ o˘ 
$$A \simeq \lexp p j^{= h}_{!*} \LC[d-h] \simeq i^{h'}_* \lexp {\tilde p} j^{h \leq h'}_{!} 
\Bigl ( \lexp p j^{h \leadsto h'}_{!*} \LC[d-h] \Bigr )$$
est un faisceau pervers sans torsion. Ainsi la nullitÈ de cette flËche dÈcoule du cas 
des coefficients $\overline \Qm_l$ traitÈ prÈcÈdemment.
\end{proof}

ConsidÈrons alors
$$\xymatrix{
0 \ar[r] & \Fil^{r-1}_!(\Psi_{\IC}) \ar[r] & \Fil^r_!(\Psi_{\IC}) \ar[r] & 
\gr^r_{!} (\Psi_\IC) \ar[r] & 0 \\
0 \ar[r] & \bigoplus_{1 \leq g \leq d} \bigoplus_{\varrho \in \scusp_v(g)} 
\Fil^{r-1}_{!,\varrho} (\Psi_{\IC}) 
\ar@{=}[u] \ar[r] & P_{\varrho_0} \ar@{-->}[r] \ar@{-->}[u] & \gr^r_{!,\varrho_0} (\Psi_\IC) \ar[u]
\ar[r] & 0
}$$
D'aprËs les propositions \ref{prop-hyp1} et \ref{prop-hyp2}, $\gr^r_{!,\varrho_0} (\Psi_\IC)$
admet une filtration dont les graduÈs sont sans torsion et isomorphes ‡ des $p$-extensions
intermÈdiaires de systËmes locaux de type $\varrho_0$. De mÍme pour tout $\varrho$,
le faisceau pervers $\Fil^{r-1}_{!,\bar \tau} (\Psi_{\IC})$ admet une filtration dont les 
graduÈs sont sans torsion et isomorphes ‡ des $p$-extensions intermÈdiaires de
systËmes locaux de type $\varrho$. Il rÈsulte alors du lemme prÈcÈdent que 
$P_{\varrho_0}$ s'Ècrit comme une somme directe
$$P_{\varrho_0} \simeq  \Fil^{r}_{!,\varrho_0} (\Psi_{\IC}) \oplus 
\bigoplus_{\varrho \not \simeq \varrho_0}  \Fil^{r-1}_{!,\varrho} (\Psi_{\IC}).$$
En rÈpÈtant le raisonnement prÈcÈdent pour tous les $\varrho$, on en dÈduit le rÈsultat.
\end{proof}

\subsection{Preuve de la proposition \ref{prop-principal}}

CommenÁons par le lemme suivant.

\begin{lemm}
Pour tout $1 \leq h \leq d$ et pour tout $\varrho$, les $p$ et $p+$ extensions intermÈdiaires
de $j^{1 \leq h,*} \Psi_{\IC,\varrho}$ sont les mÍmes, i.e.
$$\lexp p j^{1 \leq h}_{!*} j^{1 \leq h,*} \Psi_{\IC,\varrho} =
\lexp {p+} j^{1 \leq h}_{!*} j^{1 \leq h,*} \Psi_{\IC,\varrho}.$$
\end{lemm}

\begin{proof}
Rappelons que le morphisme d'adjonction
$$\lexp {p+} j^{1 \leq h}_! j^{1 \leq h,*} \Psi_{\IC,\varrho} \longrightarrow \Psi_{\IC,\varrho}$$
a pour image dans $\lexp p \CC$, le faisceau pervers $\Fil^{h}_{!} (\Psi_{\IC,\varrho})$
et pour conoyau $\lexp p \hi^0 i^{h+1,*} \Psi_{\IC,\varrho}$. D'aprËs \cite{boyer-LT} ce dernier
est libre ce qui nous fournit une surjection
$$\Fil^h_{!} (\Psi_{\IC,\varrho}) \twoheadrightarrow \lexp p j^{1 \leq h}_{!*}  j^{1 \leq h,*} 
\Psi_{\IC,\varrho}$$
puisque le noyau du morphisme d'adjonction prÈcÈdent est ‡ support dans
$X^{\geq h+1}_{\IC,\bar s}$. Soit alors le poussÈ en avant $P$
$$\xymatrix{
0 \ar[r] & \Fil^h_{!} (\Psi_{\IC,\varrho}) \ar[r] \ar@{->>}[d] & \Psi_{\IC,\varrho} \ar[r] 
\ar@{-->>}[d] & 
\Psi_{\IC,\varrho} / \Fil^h_{!} (\Psi_{\IC,\varrho}) \ar[r] \ar@{=}[d] & 0 \\
0 \ar[r] & \lexp p j^{1 \leq h}_{!*}  j^{1 \leq h,*}  \Psi_{\IC,\varrho}  \ar@{-->}[r] & P \ar[r] & 
\Psi_{\IC,\varrho} / \Fil^h_{!} (\Psi_{\IC,\varrho}) \ar[r] & 0.
}$$
Notons que $j^{1 \leq h,*} \Psi_{\IC,\varrho} \simeq j^{1 \leq h,*} P$ et comme
le noyau de $\Psi_{\IC,\varrho} \twoheadrightarrow P$ est ‡ support dans 
$X^{\geq h+1}_{\IC,\bar s}$, alors pour tout $\delta >0$ on a
$$\lexp p \hi^\delta i^{h+1,!} \Psi_{\IC,\varrho} \simeq \lexp p \hi^\delta i^{h+1,!} P.$$
Rappelons, cf. par exemple \cite{boyer-torsion}, que pour tout faisceau pervers $Q$ sans torsion, 
le conoyau du morphisme d'adjonction $Q \longrightarrow j_* j^* Q$ est isomorphe
‡ $\lexp p \hi^1 i^! Q$. On en dÈduit ainsi que
\begin{itemize}
\item $j^{1 \leq h}_* j^{1 \leq h,*} P \simeq j^{1 \leq h}_* j^{1 \leq h,*} \Psi_{\IC,\varrho}$ et 

\item les conoyaux des morphismes d'adjonction
$$P \longrightarrow j^{1 \leq h}_* j^{1 \leq h,*} P \quad \hbox{ et } \quad
\Psi_{\IC,\varrho} \longrightarrow j^{1 \leq h}_* j^{1 \leq h,*} \Psi_{\IC,\varrho}$$
sont isomorphes.
\end{itemize}
On obtient ainsi une surjection dans $\lexp p \CC$
$$P \twoheadrightarrow \CoFil_{!,h}(\Psi_{\IC,\varrho}),$$
o˘ on rappelle que $\CoFil_{!,\bullet}(\Psi_{\IC,\varrho})$ 
est la cofiltration de stratification de $\Psi_{\IC,\varrho}$ de la fin du \S \ref{para-rappel}.
Comme le socle de $P \otimes_{\overline \Zm_l} \overline \Qm_l$ ne contient aucun faisceau pervers
‡ support dans $X^{\geq h+1}_{\IC,\bar s}$,
on en dÈduit que la surjection prÈcÈdente est aussi injective et donc
$$P \simeq \CoFil_{!,h}(\Psi_{\IC,\varrho}).$$
Or dans $\lexp p \CC(X_{\IC,\bar s},\overline \Zm_l)$, la surjection 
$\Fil^h_{!} (\Psi_{\IC,\varrho}) \twoheadrightarrow 
\lexp p  j^{1 \leq h}_{!*}  j^{1 \leq h,*} \Psi_{\IC,\varrho}$, se dualise
dans $\lexp {p+} \CC(X_{\IC,\bar s},\overline \Zm_l)$ en une injection 
$$\lexp {p+} j^{1 \leq h}_{!*}  j^{1 \leq h,*} 
\Psi_{\IC,\varrho}  \hookrightarrow  \CoFil_{!,h}(\Psi_{\IC,\varrho})$$
i.e. en une injection dans $\lexp p \CC(X_{\IC,\bar s},\overline \Zm_l)$ dont le conoyau est sans torsion. 
On se retrouve alors dans la situation suivante:
$$\xymatrix{
0 \ar[r] & \lexp p j^{1 \leq h}_{!*}  j^{1 \leq h,*} \Psi_{\IC,\varrho} \ar[r] &
\CoFil_{!,h}(\Psi_{\IC,\varrho}) \ar[r] & Z \ar[r] & 0 \\
0 \ar[r] & \lexp {p+} j^{1 \leq h}_{!*}  j^{1 \leq h,*} \Psi_{\IC,\varrho} \ar[r] &
\CoFil_{!,h}(\Psi_{\IC,\varrho}) \ar[r] \ar@{=}[u] & Z' \ar[r] & 0 
}$$
o˘ $Z$ et $Z'$ sont des faisceaux pervers sans torsion ‡ support dans 
$X^{\geq h+1}_{\IC,\bar s}$. Comme il n'y a pas de flËche non nulle entre
$\lexp {p+} j^{1 \leq h}_{!*}  j^{1 \leq h,*} \Psi_{\IC,\varrho}$ et un faisceau pervers
sans torsion ‡ support dans $X^{\geq h+1}_{\IC,\bar s}$, on obtient alors une flËche de
$$\lexp {p+} j^{1 \leq h}_{!*}  j^{1 \leq h,*} \Psi_{\IC,\varrho} \longrightarrow 
\lexp {p} j^{1 \leq h}_{!*}  j^{1 \leq h,*} \Psi_{\IC,\varrho}$$
lesquels sont alors isomorphes comme annoncÈ.
\end{proof}

Revenons ‡ prÈsent ‡ la preuve de la proposition \ref{prop-principal}. Pour $\varrho$ fixÈ,
$g:=g(\varrho)$ est l'indice $g'$ minimal tel que $j^{= g',*} \Psi_{\IC,\varrho}$ est non nul. 
Avec la notation \ref{nota-L5}, on a alors
$$j^{\geq g,*} \Psi_{\IC,\varrho} \otimes_{\overline \Zm_l} \overline \Qm_l \simeq 
\bigoplus_{\bar \tau \in \RC_{\overline \Fm_l}(g,\varrho)} \LC_{\overline \Qm_l}(\bar \tau).$$
Les graduÈs de la filtration de stratification exhaustive de $\Fil^g_{!} (\Psi_{\IC,\varrho})$ sont,
d'aprËs les propositions \ref{prop-hyp1} et \ref{prop-hyp2}, et avec les notations de la fin 
du \S \ref{para-HT},
de la forme $\lexp p j^{\geq tg}_{!*} \HT_\Gamma(\pi_v,\Pi_t)(n)$ o˘ 
\begin{itemize}
\item $\pi_v$ est une $\overline \Qm_l$-reprÈsentation irrÈductible cuspidale entiËre de type $\varrho$,

\item $t$ varie de $1$ ‡ $\frac{d}{g}$ 

\item et
\begin{itemize}
\item $\Pi_t$ est une reprÈsentation de $GL_{tg}(F_v)$,

\item $n$ est un entier,

\item $\Gamma$ est un rÈseau stable
\end{itemize}
qu'il est, d'aprËs le lemme \ref{lem-principal}, inutile ici de prÈciser.
\end{itemize}
Pour $1 \leq t \leq \frac{d}{g}$ fixÈ, on peut ainsi Ècrire
$$0 \rightarrow A_{\varrho}(t) \longrightarrow \Fil^g_{!} (\Psi_{\IC,\varrho}) \longrightarrow
B_{\varrho}(t) \rightarrow 0$$
o˘ $A_{\varrho}(t)\otimes_{\overline \Zm_{l}} \overline \Qm_{l}$ rassemble les constituants irrÈductibles de  
$ \Fil^g_{!} (\Psi_{\IC,\bar†\pi} ) \otimes_{\overline \Zm_{l}} \overline \Qm_{l}$ ‡ support dans
$X^{\geq tg+1}_{\IC,\bar s}$. En particulier, avec la notation \ref{nota-HT}, il existe, avec les notations
prÈcÈdentes, $\Pi_t,n$ dÈpendant de $\pi_v$, tels que
$$\bigoplus_{\pi_v \in \scusp(\varrho)} \lexp p j^{=tg}_{!*} \HT(\pi_v,\Pi_t)(n) \hookrightarrow
B_{\varrho}(t) \otimes_{\overline \Zm_{l}} \overline \Qm_{l} $$
et dont le quotient est ‡ support dans $X^{\leq tg-1}_{\IC,\bar s}$. On note $B_{\varrho}(=t)$ (resp.
$\lexp p j^{=tg}_{!*} \HT_\Gamma(\pi_v,\Pi_t)(n)$)
le rÈseau de la somme directe prÈcÈdente dÈcoupÈ par $B_{\varrho}(t)$, i.e. les tirÈs en arriËre
$$\xymatrix{
 \lexp p j^{=tg}_{!*} \HT(\pi_v,\Pi_t)(n)  \ar@{^{(}->}[r] &
\bigoplus_{\pi_v \in \scusp(\varrho)} \lexp p j^{=tg}_{!*} \HT(\pi_v,\Pi_t)(n) \ar@{^{(}->}[r] &
B_{\varrho}(t) \otimes_{\overline \Zm_{l}} \overline \Qm_{l} \\
\lexp p j^{=tg}_{!*} \HT_\Gamma(\pi_v,\Pi_t)(n)  \ar@{^{(}-->}[r]  \ar@{^{(}-->}[u] &
B_{\varrho}(=t) \ar@{^{(}-->}[r] \ar@{^{(}-->}[u] & B_{\varrho}(t). \ar@{^{(}->}[u]
}$$
Les monomorphismes de la ligne du bas sont alors stricts, i.e. les conoyaux sont libres.
On a alors le diagramme commutatif suivant
$$\xymatrix{
0 \ar[r] & A_{\varrho}(t) \ar[r] \ar@{=}[d] & \Fil^g_{!} (\Psi_{\IC,\varrho}) \ar[r] \ar@{^{(}->}[d] & B_{\varrho}(t) 
\ar[r] \ar@{^{(}->}[d] & 0 \\
0 \ar[r] & A_{\varrho}(t) \ar[r]  & \Fil^{tg}_{!} (\Psi_{\IC,\varrho}) \ar[r] \ar@{->>}[d]  & 
\lexp {p} j^{1 \leq tg}_{!*}  j^{1 \leq tg,*} \Psi_{\IC,\varrho} \ar[r] \ar@{->>}[d] & 0 \\
& & F \ar@{=}[r] & Q
}$$
Pour tout $\pi_v \in \scusp(\varrho)$, en composant les inclusions de $\lexp {p+} \CC$
$$\lexp p j^{= tg}_{!*} \HT_\Gamma(\pi_v,\Pi_t)(n) \hookrightarrow B_{\varrho}(t) \hookrightarrow 
\lexp {p} j^{1 \leq tg}_{!*}  j^{1 \leq tg,*} \Psi_{\IC,\varrho}$$
on obtient un monomorphisme strict de $\FC$ 
$$\lexp p j^{= tg}_{!*} \HT_\Gamma(\pi_v,\Pi_t)(n) \hookrightarrow 
\lexp {p} j^{1 \leq tg}_{!*}  j^{1 \leq tg,*} \Psi_{\IC,\varrho},$$
i.e. le conoyau est sans torsion. Or par ailleurs, comme le terme de droite de cette inclusion est 
aussi Ègal ‡ $\lexp {p+} j^{1 \leq tg}_{!*}  j^{1 \leq tg,*} \Psi_{\IC,\varrho}$
on a une monomorphisme strict
\addtocounter{smfthm}{1}
\begin{equation} \label{eq-mono2}
\lexp {p+} j^{=tg}_{!*} \HT_{\Gamma'}(\pi_v,\Pi_t)(n) \hookrightarrow
\lexp {p} j^{1 \leq tg}_{!*}  j^{1 \leq tg,*} \Psi_{\IC,\varrho}
\end{equation}
pour un certain rÈseau stable. En rÈsumÈ, on a
$$\xymatrix{
&\lexp {p} j^{=tg}_{!*} \HT_{\Gamma}(\pi_v,\Pi_t)(n) \ar@{^{(}->}[d]^{strict} \\
\lexp {p+} j^{=tg}_{!*} \HT_{\Gamma'}(\pi_v,\Pi_t)(n) \ar@{^{(}->}[r]_{strict} \ar@{^{(}-->}[ur] \ar[dr]_{nul} &
\lexp {p} j^{1 \leq tg}_{!*}  j^{1 \leq tg,*} \Psi_{\IC,\varrho} \ar@{->>}[d] \\ & Q
}$$
Par ailleurs sur $\overline \Qm_l$, d'aprËs \cite{boyer-invent2}, la multiplicitÈ de
$\lexp p j^{=tg} \HT (\pi_v,\Pi_t)(n)$ dans $\lexp {p} j^{1 \leq tg}_{!*}  j^{1 \leq tg,*} \Psi_{\IC,\varrho}$
est Ègale ‡ $1$ et donc, est nulle dans $Q$ de sorte que, cf. la proposition 1.1.8 de \cite{quasi-ab}, 
le monomorphisme strict (\ref{eq-mono2})
se factorise en un monomorphisme strict
$$\lexp {p+} j^{=tg}_{!*} \HT_{\Gamma'}(\pi_v,\Pi_t)(n) \hookrightarrow 
\lexp {p} j^{=tg}_{!*} \HT_{\Gamma}(\pi_v,\Pi_t)(n).$$
En appliquant $j^{=tg,*}$, on obtient en particulier $\Gamma'=\Gamma$ puis comme le composÈ
$$\lexp {p+} j^{=tg}_{!*} \HT_{\Gamma}(\pi_v,\Pi_t)(n) \hookrightarrow 
\lexp {p} j^{=tg}_{!*} \HT_{\Gamma}(\pi_v,\Pi_t)(n) \hookrightarrow 
\lexp {p+} j^{=tg}_{!*} \HT_{\Gamma}(\pi_v,\Pi_t)(n)$$
est un isomorphisme, on obtient bien 
$\lexp p j^{= tg}_{!*} HT_\Gamma(\pi_v,\Pi_t)(n) \simeq \lexp {p+} 
j^{= tg}_{!*} HT_\Gamma(\pi_v,\Pi_t)(n)$.

\section{Extensions intermÈdiaires: cas non supercuspidal}
\label{para-cuspi}

Soit $\varrho$ une $\overline \Fm_l$-reprÈsentation irrÈductible supercuspidale de $GL_g(F_v)$
et $\bar \tau_{\varrho,t}$ la reprÈsentation de $D_{v,tg}^\times$
de type $\varrho$ associÈe ‡ la $\overline \Fm_l$-reprÈsentation superSpeh $\speh_t(\varrho)$,
cf. le thÈorËme \ref{theo-jl}.

Pour tout $\overline \Qm_l$-reprÈsentation irrÈductible entiËre $\tau_v$
de $\bar \tau_{\varrho,t}$-type $-1$, d'aprËs le thÈorËme \ref{theo-principal}, pour tout rÈseau stable
$\Gamma$ de $\tau_v$, le systËme local
$\LC_{\Gamma}(\tau_v)$ n'admet qu'une seule extension intermÈdiaire, i.e.
$$\lexp p j^{=tg}_{!*} \LC_{\Gamma}(\tau_v)[d-tg] \simeq \lexp {p+} j^{=tg}_{!*}
\LC_{\Gamma}(\tau_v)[d-tg].$$
En utilisant le triangle distinguÈ, cf. \cite{juteau} 2.42-2.46,
\begin{multline*}
\lexp p j^{=tg}_{!*} \LC_\Gamma(\tau_v)[d-tg] \rightarrow 
\lexp {p+} j^{=tg}_{!*} \LC_\Gamma(\tau_v)[d-tg] \rightarrow \\
\lexp p i^{tg+1}_* \hi^0_{tors} i^{tg+1,*} j^{\geq tg}_* \LC_\Gamma(\tau_v)[d-tg] \leadsto
\end{multline*}
la torsion de $\lexp p i^{tg+1}_* \hi^0 i^{tg+1,*}j^{\geq tg}_* \LC_\Gamma(\tau_v)[d-tg]$ 
est nulle. Le but de ce paragraphe est d'expliciter la $l$-torsion de ce dernier dans
le cas o˘ $\tau_v$ est de $\bar \tau_{\varrho,t}$-type $i \geq 0$.

\subsection{RÈseaux d'induction d'aprËs \texorpdfstring{\cite{boyer-repmodl}}{Lg}}
\label{para-RI}

Pour $\pi$ une reprÈsentation irrÈductible cuspidale entiËre de $GL_g(K)$,
comme, d'aprËs \ref{prop-red-modl},
sa rÈduction modulo $l$, notÈe $\varrho$, est irrÈductible, on en dÈduit
qu'‡ isomorphismes prËs, $\pi$
possËde un unique rÈseau stable, cf. par exemple \cite{bellaiche-ribet} proposition 3.3.2 et la remarque qui suit.

\begin{defi} \phantomsection \label{defi-RI} (cf. \cite{boyer-repmodl})
…tant donnÈ un rÈseau de $\st_t(\pi)$, la surjection (resp. l'injection)
$$\st_t(\pi) \overrightarrow{\times} \pi  \twoheadrightarrow \st_{t+1}(\pi), 
i$$
induit un rÈseau de $\st_{t+1}(\pi)$ de sorte que par rÈcurrence on dispose d'un rÈseau  
$RI_{\overline \Zm_l,-}(\pi,t)$ que l'on qualifie de \textit{rÈseau d'induction}. On note alors
$$RI_{\overline \Fm_l,-}(\pi,t):= RI_{\overline \Zm_l,-}(\pi,t) \otimes_{\overline \Zm_l} \overline \Fm_l, 
$$
\end{defi}

\begin{prop} \phantomsection \label{prop-defi-Vk} 
(cf. \cite{boyer-repmodl} propositions 3.2.2 et 3.2.7)
Pour tout $0 \leq k \leq \lg_\varrho(s)$, il existe une sous-reprÈsentation $V_{\varrho,-}(s;k)$ 
de longueur $k$ de $RI_{\overline \Fm_l,-}(\pi,s)$
$$(0)=V_{\varrho,\pm}(s;0) \varsubsetneq V_{\varrho,\pm}(s;1) \varsubsetneq \cdots \varsubsetneq
V_{\varrho,\pm}(s;\lg_\varrho(s))= RI_{\overline \Fm_l,-}(\pi,s),$$
dÈfinie de sorte que l'image de $V_{\varrho,-}(s;k)$ dans le groupe de Grothendieck est telle
que tous ses constituants irrÈductibles sont
de $\varrho$-niveau strictement plus grand que n'importe quel constituant irrÈductible de
$W_{\varrho,-}(s;k):= V_{\varrho,-}(s;\lg_\varrho(s))/V_{\varrho,-}(s;k)$.
\end{prop}

\begin{nota} \phantomsection \label{nota-V}
Une reprÈsentation irrÈductible $\varrho$ Ètant fixÈe ainsi qu'un entier $s$,
pour $k\geq 0$ tel que $m(\varrho)l^k \leq s$, on note:
\begin{itemize}
\item $\underline{\delta_k}=(0,\cdots,0,1,0,\cdots) \in \IC_\varrho(s)$ et

\item pour tout $t$ tel que $m(\varrho)l^kt \leq s$,
$V_{\varrho,-}(s,\geq t.\underline{\delta_k})$ le sous-espace $V_{\varrho,-}(s,\lg_\varrho(s))$ 
dÈfini ci-dessus tel que
tous les constituants irrÈductibles de $V_{\varrho,-}(s,\lg_\varrho(s))$ sont de $\varrho$-niveau plus
grand ou Ègal ‡ $t.\underline{\delta_k}$.
\end{itemize}
\end{nota}

\subsection{RÈduction modulo \texorpdfstring{$l$}{Lg} d'un faisceau pervers d'Harris-Taylor}
\label{para-nota}

Pour Ètudier la $l$-torsion de 
$\lexp p i^{tg+1}_* \hi^0 i^{tg+1,*}j^{\geq tg}_* \LC_{\Gamma}(\tau_v)[d-tg]$, nous utiliserons
le foncteur de rÈduction modulaire
$$\Fm(-):=\overline \Fm_l \otimes_{\overline \Zm_l}^\Lm (-).$$ 
Rappelons que 
ce dernier ne commute pas aux foncteurs de troncations et que d'aprËs les Èquations 
2.54-2.61 de \cite{juteau}, on a
$$\Fm \lexp p j_{!*} \rightarrow \lexp p j_{!*} \Fm \rightarrow \hi^{-1} \Fm \lexp p i_*
\lexp p \hi^0_{tors} i^* j_*[1] \leadsto$$
$$\lexp p j_{!*} \Fm \rightarrow \Fm \lexp {p+} j_{!*} \rightarrow \hi^0 \Fm \lexp p i_*
\lexp p \hi^0_{tors} i^*j_* \leadsto$$
En revanche, dans le cas o˘ $\lexp p j_!=\lexp {p+} j_!$, i.e. en utilisant 
$$\lexp p j_! \rightarrow \lexp {p+} j_! \rightarrow \lexp p i_* \hi^{-1}_{tors} i^* j_* [1] \leadsto$$
quand $ \lexp p i_* \hi^{-1} i^* j_*$ est libre, alors le triangle distinguÈ
$$\Fm  \lexp p j_! \rightarrow \lexp p j_! \Fm \rightarrow \hi^{-1} \Fm \lexp p i_* \lexp p
\hi^{-1}_{tors}i^*j_* [2] \leadsto$$
nous donne que $\Fm$ et $\lexp p j_!$ commutent. Nous allons utiliser cette propriÈtÈ
avec les $j^{\geq tg}$. Rappelons les notations du \S \ref{para-GL-l}:
\begin{itemize}
\item $\varrho:=\varrho_{-1}$ dÈsignera une $\overline \Fm_l$-reprÈsentation irrÈductible 
supercuspidale de $GL_g(F_v)$ et on notera aussi $g_{-1}:=g$.

\item Pour tout $u \geq 0$, selon les notations de \ref{nota-rhoi}, on note 
$\varrho_u=\st_{m(\varrho_{-1})l^u}(\varrho_{-1})$ 
la reprÈsentation irrÈductible cuspidale de $GL_{g_u}(F_v)$ avec $g_u=g m(\varrho_{-1})l^u$.

\item Pour tout $u \geq -1$, on se fixe une reprÈsentation irrÈductible cuspidale $\pi_u$
de $GL_{g_u}(F_v)$ dont la rÈduction modulo $l$ est isomorphe ‡ $\varrho_u$ et 

\item enfin on notera $s_u:=\lfloor \frac{d}{g_u} \rfloor$.
\end{itemize}

\begin{nota}
Pour tout $t \geq 1$, on note 
$$\bar \tau_{\varrho_u,t}:=r_l \bigl ( \pi_u[t]_D \bigr ).$$
\end{nota}

\rem on rappelle, cf. le \S \ref{para-repD}, que $\bar \tau_{\varrho_{-1},t}$ est irrÈductible.

\begin{prop} (cf. la proposition \ref{prop-repD}) \\
Pour $u \geq 0$, on a l'ÈgalitÈ suivante dans le groupe de Grothendieck
\addtocounter{smfthm}{1}
\begin{equation} \label{eq-redmodl}
\bar \tau_{\varrho_u,t}=l^u \sum_{i=0}^{m(\varrho)-1} \bar \tau_{\varrho,tm(\varrho)l^u} \nu^i.
\end{equation}
\end{prop}

\rem dans le cas o˘, avec les notations de \ref{defi-mrho}, $\epsilon(\varrho)=1$, la formule
(\ref{eq-redmodl}) s'Ècrit $\bar \tau_{\varrho_u,t}=l^{u+1} \bar \tau_{\varrho,t}$.
%
%

\begin{prop} \phantomsection \label{prop-dec-pervers}
Avec les notations de la proposition \ref{prop-defi-Vk}, dans le groupe de Grothendieck des
$\overline \Fm_l$-faisceaux pervers Èquivariants sur $X_{\IC,\bar s}$, on a l'ÈgalitÈ
\begin{multline*}
\Fm \Bigl ( \lexp p j_{!*}^{\geq tg_{u}} \HT(\pi_{u},\Pi_t) \Bigr )=  m(\varrho)l^{u} 
\sum_{r=0}^{s-tm(\varrho)l^u} \lexp p j_{!*}^{\geq tg_{u}+rg_{-1}} \\
\HT \bigl (\varrho,r_l(\Pi_{t}) \overrightarrow{\times} V_{\varrho_{-1}}(r+tm(\varrho_{-1})l^u ,
< \underline{\delta_u}) \bigr ) \otimes \Xi^{r\frac{g-1}{2}}.
\end{multline*}
\end{prop}

\rem dans le groupe de Grothendieck, l'induite $r_l(\Pi_t)  \overrightarrow{\times} 
V_{\varrho_{-1}}(r+tm(\varrho_{-1})l^u ,< \underline{\delta_u})$ n'intervient que par 
sa semi-simplifiÈe. Pour les mÍmes raisons, il est inutile de prÈciser les rÈseaux stables
utilisÈs pour les systËmes locaux de la formule prÈcÈdente.

\begin{proof}
Les cas $tg_u \geq d$ Ètant triviaux, on raisonne par rÈcurrence en supposant le rÈsultat
acquis pour tout $t<t' $ et on traite le cas de $t$. L'idÈe est de partir de la commutation
entre $\Fm$ et les $j^{\geq tg}_!$:
$$\Fm \bigl ( j_!^{\geq tg_u} \HT(\pi_u,\Pi_{t}) \bigr )= 
\sum_{t'=t}^{s_u} \Fm \bigl ( \lexp p j_{!*}^{\geq t'g_u} \HT(\pi_u,\Pi_t \overrightarrow{\times}
\st_{t'-t}(\pi_u) ) \bigr ) \otimes \Xi^{\frac{(t'-t)(g_u-1)}{2}}),$$
et, en posant $t(u)=tm(\varrho)l^u$
\addtocounter{smfthm}{1}
\begin{multline} \label{eq-dec-1}
\Fm \bigl ( j_!^{\geq t(u)g_{-1}} \HT(\pi_{-1},\Pi_{t}) \bigr )= \\
\sum_{t'=t(u)}^s \Fm \bigl ( \lexp p j_{!*}^{\geq t'g_{-1}} \HT(\pi_{-1}, \Pi_{t}) \overrightarrow{\times} 
\st_{t'-t(u)}(\pi_{-1}) \bigr ) \otimes \Xi^{\frac{(t'-t(u))(g_{-1}-1)}{2}}.
\end{multline}
Or on a vu que
$$\Fm j_!^{\geq tg_u} \HT(\pi_u,\Pi_t) =j_!^{\geq tg_u} \Fm \HT(\pi_u,\Pi_t) =
m(\varrho) l^u j_!^{\geq t(u)g_{-1}} \Fm \HT(\pi_{-1},\Pi_t)$$
et d'aprËs l'hypothËse de rÈcurrence, on a
\addtocounter{smfthm}{1}
\begin{multline} \label{eq-dec-2}
\sum_{t'=t+1}^{s_u}
\Fm \Bigl ( \lexp p j_{!*}^{\geq t'g_u} \HT(\pi_u,\Pi_t \overrightarrow{\times} \st_{t'-t}(\pi_u)) \Bigr ) \otimes
\Xi^{\frac{(t'-t)(g_u-1)}{2}} \\ 
= \sum_{t'=t(u)+1}^{s(u)} \lexp p j_{!*}^{\geq t'g_{-1}}
\HT(\varrho,r_l(\Pi_{t}) \overrightarrow{\times} V_{\varrho_{-1}}(t'-t(u),\geq \underline{\delta_u}) ) 
\otimes \Xi^{\frac{(t'-t(u))(g_1-1)}{2}}
\end{multline}
En soustrayant (\ref{eq-dec-2}) ‡ (\ref{eq-dec-1}), on obtient le rÈsultat.
\end{proof}

\rem la surjection
$$\lexp p j_{!}^{\geq t'g_u} \Fm \HT(\pi_u,\Pi_t{'})  \twoheadrightarrow
\Fm \lexp p j_{!*}^{\geq t'g_u} \HT(\pi_u,\Pi_{t'})$$
nous donne en outre que la suite des dimensions des
graduÈs de la filtration de stratification exhaustive est strictement croissante ce qui fixe
complËtement cette filtration.

Ainsi la $l$-torsion du quotient des $p+$ faisceaux pervers d'Harris-Taylor par
leur $p$ version, est complËtement dÈcrit par la combinatoire
de la rÈduction modulo $l$ des reprÈsentations de $GL_d(F_v)$ et de $D_{v,d}^\times$.
L'Ètude de la torsion d'ordre supÈrieure dÈcoulerait selon le mÍme schÈma de dÈmonstration,
de l'Ètude de la rÈduction modulo $l^n$ des reprÈsentations irrÈductibles
de $GL_d(F_v)$ et $D_{v,d}^\times$.

\subsection{Complexes d'induction des reprÈsentations de Steinberg}
\label{para-complexe}

La filtration de stratification du faisceau pervers 
$\Fm \bigl ( \lexp p j_{!*}^{\geq tg_{u}} \HT(\pi_{u},\Pi_{t}) \bigr )$ 
dont les graduÈs sont d'aprËs la remarque
de la fin du paragraphe prÈcÈdent, les faisceaux pervers de la proposition 
\ref{prop-dec-pervers} pris dans l'ordre inverse de la dimension de leur support 
(le sous-espace est donnÈ par le faisceau pervers dont le support est de dimension minimale),
fournit une suite spectrale qui calcule ses faisceaux de cohomologie.
Le but de ce paragraphe est de dÈcrire la combinatoire de cette suite spectrale.

\begin{nota}
Pour tout $s \geq 1$,
on note $K_{\pi}(s)^\bullet$ le complexe
$$K_\pi(s)^i=\left \{ \begin{array}{ll} 0 & \hbox{si } i \geq 0 \hbox{ ou } i<-s \\ RI_{\overline \Zm_l,-}(\pi,s+i)
\overrightarrow{\times} \speh_{-i}(\pi) & \hbox{pour } -s \leq i \leq 0 \end{array} \right.$$
\end{nota}

\rem d'aprËs la dÈfinition rappelÈe au \S \ref{para-RI} des rÈseaux d'induction 
$RI_{\overline \Zm_l,-}(\pi,s+i)$, la cohomologie de $K_{\pi}(s)^\bullet$ est nulle. 

En notant $\varrho$ la rÈduction modulo $l$ de $\pi$, on dÈfinit aussi
$$K_\varrho(s)^\bullet:=K_\pi(s)^\bullet \otimes_{\overline \Zm_l} \overline \Fm_l$$
dont la cohomologie est nulle.

\begin{defi} Pour tout $k,t \geq 0$, tels que $m(\varrho)l^kt \leq s$, avec les notations de \ref{nota-V}, on dÈfinit
\begin{itemize}
\item $K_\varrho(s,\geq t.\underline{\delta_k})^\bullet$ le sous-complexe de
$K_\varrho(s)$ dÈfini,  pour $-s \leq i \leq 0$ par
$$K_\varrho(s,\geq t.\underline{\delta_k})^i=V_\varrho(s+i,\geq t.\underline{\delta_k}) 
\overrightarrow{\times} \speh_{-i}(\varrho);$$

\item $K_\varrho(s,< t.\underline{\delta_k})^\bullet$ le quotient de $K_\varrho(s)^\bullet$ par 
$K_\varrho(s,\geq t.\underline{\delta_k})^\bullet$.
\end{itemize}
\end{defi}

\begin{prop} \phantomsection \label{prop-coho-complexe}
La cohomologie $\hi^i K_\varrho(s,<\underline{\delta_k})$ du complexe $
K_\varrho(s,<t.\underline{\delta_k})^\bullet$ est:
\begin{itemize}
\item nulle si $m(\varrho)l^k$ ne divise pas $s$;

\item pour $s=\delta m(\varrho)l^k$, elle est nulle si $i \neq -\delta$ et pour $i=-\delta$ isomorphe ‡
$\speh_{\delta}(\rho_k)$.
\end{itemize}
\end{prop}

\begin{proof}
On raisonne par rÈcurrence sur $s$ quel que soit $\varrho$; l'initialisation Ètant triviale supposons
donc le rÈsultat acquis jusqu'au rang $s-1$ et traitons le cas de $s$.

Soit $u$ tel que $m(\varrho)l^u \leq s < m(\varrho)l^{u+1}$; pour tout $0 \leq k \leq u$, on note 
$t_k \geq 1$ tel que $t_km(\varrho)l^k \leq s < (t_k+1)m(\varrho)l^k$ et on considËre la filtration 
suivante de $K_\varrho(s)^\bullet$:
\begin{multline} \label{filtration-K}
K_\varrho(s,>t_u \underline{\delta_u})^\bullet \subset K_\varrho(s,>(t_u-1) \underline{\delta_u})^\bullet
\subset \cdots \subset K_\varrho(s,>\underline{\delta_u})^\bullet \\
\subset  K_\varrho(s,>t_{u-1} \underline{\delta_{u-1}})^\bullet \subset \cdots \subset
K_\varrho(s,>\underline{\delta_{u-1}})^\bullet \\
\cdots \subset K_\varrho(s,>t_0 \underline{\delta_0})^\bullet \subset \cdots \subset 
K_\varrho(s,>\underline{\delta_0})^\bullet \subset K_\varrho(s)^\bullet
\end{multline}
On a alors les propriÈtÈs suivantes:
\begin{itemize}
\item pour $1 \leq k \leq u$, $K_\varrho(s,>t_{k-1}.\underline{\delta_{k-1}})^\bullet / 
K_\varrho(s,>\underline{\delta_k})^\bullet$
est le complexe
$$\st_{t_{k-1}}(\rho_{k-1}) \overrightarrow{\times} 
K_\varrho(s-t_{k-1}m(\varrho)l^{k-1},<\underline{\delta_0})^\bullet,$$
dont la cohomologie est, d'aprËs l'hypothËse de rÈcurrence, nulle sauf si $s=t_{k-1}m(\varrho)l^{k-1}$ auquel cas
$\hi^0$ est le seul $\hi^i$ non nul, et alors isomorphe ‡ $\st_{t_{k-1}}(\rho_{k-1})$;

\item pour $0 \leq k \leq u$ et $1 \leq t < t_k$, le complexe $K_\varrho(s,>t.\underline{\delta_k})^\bullet
/ K_\varrho(s,>(t+1).\underline{\delta_k})^\bullet$ est le complexe
$$\st_{t}(\rho_{k}) \overrightarrow{\times} K_\varrho(s-tm(\varrho)l^{k},<\underline{\delta_k})^\bullet,$$
dont la cohomologie est, d'aprËs l'hypothËse de rÈcurrence, nulle sauf si $s=t_km(\varrho)l^{k}$
auquel cas $\hi^{t-t_k}$ est le seul $\hi^i$ non nul, et alors isomorphe ‡ 
$\st_{t}(\rho_{k}) \overrightarrow{\times} \speh_{t_k-t}(\rho_k)$.
\end{itemize}
ConsidÈrons alors la suite spectrale de cohomologie $E_1^{i,j}=\hi^{i+j} gr_{-i} \Rightarrow E_\oo^{i+j}$,
associÈe ‡ la filtration \ref{filtration-K}, dont on rappelle que  l'aboutissement $E_\oo^{i+h}$ est nul.
\begin{itemize}
\item D'aprËs les propriÈtÈs prÈcÈdentes les $E_1^{i,j}$ sont connus par
rÈcurrence pour tout $i>0$. PrÈcisÈment, pour $0 \leq k \leq u$ tel que $s=m(\varrho)l^kt_k$, il existe $i_1<i_2< \cdots <i_{t_k}$ tels que
pour $1 \leq r \leq t_k$,
$$E_1^{i_r,-t_k-i_r+r}=\st_{\underline 0}(\rho_k) \overrightarrow{\times} \speh_{t_k-r}(\rho_k), \quad
\underline 0 \in \IC_{\rho_k}(r).$$

\item Par ailleurs comme les flËches $E_1^{i_r,-t_k-i_r+r} \rightarrow E_1^{i_{r+1},-t_k-i_{r+1}+r+1}$ 
sont induites par celles de $K_\varrho(s)^\bullet$, on en dÈduit, en utilisant l'hypothËse de rÈcurrence, 
que $E_1^{0,-t_k}$ est isomorphe ‡ $\speh_{t_k}(\rho_k)$, d'o˘ le rÈsultat.
\end{itemize}

\end{proof}

\appendix

\section{Rappels sur les reprÈsentations}

\renewcommand{\theequation}{\Alph{section}.\arabic{subsection}.\arabic{smfthm}}

\subsection{de \texorpdfstring{$GL_n(K)$}{Lg} ‡ coefficients dans \texorpdfstring{$\overline \Qm_{l}$}{Lg}}
\label{para-rappel-rep}

Dans la suite $K$ dÈsigne un corps local non archimÈdien dont le corps rÈsiduel est de cardinal $q$ une puissance de $p$ et on rappelle quelques notations de \cite{boyer-invent2} sur
sur les reprÈsentations admissibles de $GL_n(K)$ ‡ coefficients dans $\overline \Qm_l$ o˘ $l$ un nombre premier distinct de $p$.

\begin{nota}
Une racine carrÈe $q^{\frac{1}{2}}$ de $q$ dans $\overline \Qm_l$ Ètant fixÈe,
pour $k \in \frac{1}{2} \Zm$, nous noterons
$\pi\{ k \}$ la reprÈsentation tordue de $\pi$ de sorte que l'action d'un ÈlÈment $g \in GL_n(K)$ est donnÈe par $\pi(g) \nu(g)^k$ avec
$\nu: g \in GL_n(K) \mapsto q^{-\val (\det g)}$.
\end{nota}

\begin{defis}
\begin{itemize}
\item Soit $P=MN$ un parabolique de $GL_n$ de LÈvi $M$ et de radical unipotent $N$.
On note $\delta_P:P(K) \rightarrow R^\times$ l'application dÈfinie par
$$\delta_P(h)=|\det (\ad(h)_{|\lie N})|^{-1}.$$

\item Pour $(\pi_1,V_1)$ et $(\pi_2,V_2)$ des $R$-reprÈsentations de respectivement $GL_{n_1}(K)$ et $GL_{n_2}(K)$,
et $P$ un parabolique de $GL_{n_1+n_2}$ de Levi $M=GL_{n_1} \times GL_{n_2}$ et de radical unipotent $N$,
$$\pi_1 \times_{P} \pi_2$$
dÈsigne l'induite parabolique normalisÈe de $P(K)$ ‡ $GL_{n_1+n_2}(K)$ de $\pi_1 \otimes \pi_2$ c'est ‡ dire
l'espace des fonctions $f:GL_{n_1+n_2}(K) \rightarrow V_1 \otimes V_2$ telles que
$$f(nmg)=\delta_P^{-1/2}(m) (\pi_1 \otimes \pi_2)(m) \Bigl ( f(g) \Bigr ),
\quad \forall n \in N, ~\forall m \in M, ~ \forall g \in GL_{n_1+n_2}(K).$$

\rem en particulier si $P$ est standard alors $\pi_1 \times_P \pi_2$ est l'induite \og classique \fg{} de
$\pi_1\{ n_2/2 \} \otimes \pi_2 \{ -n_1/2\}$.

\item \textit{Foncteurs de Jacquet}: pour $\pi$ une $R$-reprÈsentation admissible de $GL_n(K)$,
l'espace des vecteurs $N(K)$-coinvariants est stable sous l'action de
$M(K) \simeq P(K)/N(K)$. On notera $J_P(\pi)$ cette reprÈsentation tordue par $\delta_P^{-1/2}$.
\end{itemize}
\end{defis}

\begin{notas} Dans le cas o˘ le parabolique est standard de Levi
$GL_{r_1} \times GL_{r_2} \times \cdots \times GL_{r_k}$, on le notera
$P_{r_1,\cdots,r_k}$ et $\times$ dÈsignera $\times_P$.
\end{notas}

\begin{defis}
Une reprÈsentation $\pi$ de $GL_n(K)$ est dite \textit{cuspidale} si elle n'est pas un sous-quotient d'une induite parabolique propre.

\medskip

Soient $g$ un diviseur de $d=sg$ et $\pi$ une reprÈsentation cuspidale
irrÈductible de $GL_g(K)$. L'induite parabolique
$$\pi\{ \frac{1-s}{2} \} \times \pi\{\frac{3-s}{2} \} \times \cdots \times \pi\{ \frac{s-1}{2} \}$$
possËde
\begin{itemize}
\item un unique quotient irrÈductible notÈ $\st_s(\pi)$; c'est une reprÈsentation de Steinberg gÈnÈralisÈe.

\item une unique sous-reprÈsentation irrÈductible notÈe $\speh_s(\pi)$; c'est une reprÈsentation de Speh gÈnÈralisÈe.
\end{itemize}
%
%
\end{defis}

Afin d'Èviter d'avoir ‡ Ècrire systÈmatiquement toutes ces torsions, on introduit les notations suivantes.

\begin{notas} Un entier $g \geq 1$ Ètant fixÈ, pour $\pi_1$ et $\pi_2$ des reprÈsentations de respectivement $GL_{t_1g}(K)$ et $GL_{t_2g}(K)$, on notera
$$\pi_1 \overrightarrow{\times} \pi_2 =\pi_1 \{ -\frac{t_2}{2} \} \times \pi_2 \{ \frac{t_1}{2} \} 
$$
\end{notas}
%
%

\subsection{de \texorpdfstring{$GL_d(K)$}{Lg} mod \texorpdfstring{$l$}{Lg}}
\label{para-GL-l}

On rappelle que $l$ et $p$ dÈsignent des nombres premiers distincts et que $q$ est une puissance de $p$. On note
$e_l(q)$ l'ordre de l'image de $q$ dans $\Fm_l^\times$.
Afin de simplifier la lecture, dans la suite on utilisera la lettre $\pi$ pour dÈsigner
une $\overline \Qm_l$-reprÈsentation entiËre et les lettres $\varrho$ et $\rho$ pour
des $\overline \Fm_l$-reprÈsentations.

\begin{defi} \phantomsection \label{defi-mrho}
Une reprÈsentation $\varrho$ de $GL_n(K)$ est dite 
\begin{itemize}
\item \emph{cuspidale} si pour tout sous-groupe parabolique propre $P$ de $GL_n(K)$, 
$J_{P}(\varrho)$ est nul.

\item  Elle sera dite \emph{supercuspidale} si elle n'est pas un sous-quotient d'une induite
parabolique propre.
\end{itemize}
\end{defi}

\begin{prop} \label{prop-red-modl} (cf. \cite{vigneras-livre} III.5.10)
La rÈduction modulo $l$ d'une $\overline \Qm_l$-reprÈsentation irrÈductible cuspidale entiËre
de $GL_g(K)$ est irrÈductible cuspidale.
\end{prop}

\begin{prop} \cite{dat-jl} \S 2.2.3 \\
Soit $\pi$ une reprÈsentation irrÈductible cuspidale entiËre. Alors pour tout $s \geq 1$, la rÈduction 
modulo $l$ de $\speh_{s}(\pi)$ est irrÈductible.
\end{prop}

\begin{defi} Une $\overline \Fm_l$-reprÈsentation irrÈductible est dite $l$-Speh (resp. $l$-superSpeh)
si c'est la rÈduction modulo $l$ d'une $\overline \Qm_l$-reprÈsentation entiËre $\speh_s(\pi)$
pour $\pi$ irrÈductible cuspidale (resp. et dont la rÈduction modulo $l$ de $\pi$ est supercuspidale). 
\end{defi}

\begin{nota}
On notera  $\epsilon(\varrho)$ le cardinal de la droite de Zelevinski de
$\varrho$, i.e. de l'ensemble des classes d'Èquivalence $\{ \varrho\{ i\} ~/~ i \in \Zm \}$.
 On pose alors  cf. \cite{vigneras-induced} p.51
$$m(\varrho)=\left\{ \begin{array}{ll} \epsilon(\varrho), & \hbox{si } \epsilon(\varrho)>1; \\ l, & \hbox{sinon.} \end{array} \right.$$
\end{nota}

\rem $\epsilon(\varrho)$ est un diviseur de $e_l(q)$.

\begin{defi} \phantomsection \label{defi-nd}
…tant donnÈ un multi-ensemble $\underline s=\{ \rho_1^{n_1},\cdots,\rho_r^{n_r} \}$
de reprÈsentations cuspidales, on note d'aprËs
\cite{vigneras-induced} V.7, $\st(\underline s)$ l'unique reprÈsentation non dÈgÈnÈrÈe de l'induite
$$\rho(\underline s):=\overbrace{\rho_1 \times \cdots \times \rho_1}^{n_1} \times \cdots \times \overbrace{\rho_r \times \cdots \times \rho_r}^{n_r}.$$
\end{defi}

\rem d'aprËs \cite{vigneras-induced} V.7, toutes les reprÈsentations non dÈgÈnÈrÈes sont de cette forme.

\begin{nota} \phantomsection \label{nota-st}
Pour $\rho$ une reprÈsentation irrÈductible cuspidale et $s \geq 1$, on note
$\underline s(\rho)$ le multi-segment $\{ \rho, \rho \{ 1 \}, \cdots , \rho\{ s-1 \} \}$ et comme dans
\cite{vigneras-induced} V.4, $\st_s(\rho):=\st(\underline s(\rho))$.
\end{nota}

\begin{prop} \cite{vigneras-induced} V.4
Soit $\varrho$ une reprÈsentation irrÈductible cuspidale.
La reprÈsentation non dÈgÈnÈrÈe $\st_s(\varrho)$ est cuspidale si et seulement $s=1$ ou $m(\varrho)l^k$ pour $k \geq 0$.
\end{prop}

\rem d'aprËs \cite{vigneras-livre} III-3.15 et 5.14, toute reprÈsentation irrÈductible cuspidale est de la forme
$\st_s(\varrho)$ pour $\varrho$ irrÈductible supercuspidale et $s=1$ ou de la forme $m(\varrho)l^k$ avec 
$k \geq 0$. 

\begin{nota} \phantomsection \label{nota-rhoi}
Soit $\varrho$ une reprÈsentation irrÈductible cuspidale de $GL_g(K)$; on note 
$\varrho_{-1}=\varrho$ et pour tout $i \geq 0$, $\varrho_i=\st_{m(\varrho)l^i}(\varrho)$.
\end{nota}

\begin{defi}
On dira d'une $\overline \Qm_l$-reprÈsentation irrÈductible cuspidale entiËre qu'elle est
de type $\varrho$ si, ‡ torsion par un caractËre non ramifiÈ prËs, 
sa rÈduction modulo $l$ est de la forme $\varrho_i$ pour $i \geq -1$.
\end{defi}

\begin{nota}
Soit $s \geq 1$ un entier et $\varrho$ une reprÈsentation irrÈductible cuspidale de $GL_g(K)$. Soit
$\IC_\varrho(s)$ l'ensemble des suites $(m_{-1},m_0,\cdots)$ ‡ valeurs dans $\Nm$ telles que
$$s=m_{-1}+m(\varrho)\sum_{k=0}^{+\oo} m_k l^k.$$
On notera $\lg_\varrho(s)$ le cardinal de $\IC_{\varrho}(s)$.
\end{nota}

\begin{defi} \phantomsection \label{defi-spm}
Pour $\underline i=(i_{-1},i_0,\cdots) \in \IC_\varrho(s)$, on dÈfinit
$$\st_{\underline i}(\varrho):= \st_{\underline{i_{-1}^-}}(\varrho_{-1}) \times
\st_{\underline{i_{0}^-}}(\varrho_0) \times \cdots \times \st_{\underline{i_u^-}}(\varrho_u)$$
o˘ $i_k=0$ pour tout $k >u$ et o˘ les $\varrho_i$ sont dÈfinis en \ref{nota-rhoi}.
\end{defi}

\begin{theo} \phantomsection \label{theo-ss-quotient}
Soit $\pi$ une $\overline \Qm_l$-reprÈsentation irrÈductible cuspidale entiËre
de $GL_g(K)$ et $\varrho$ sa rÈduction modulo $l$.
Dans le groupe de Grothendieck des $\overline \Fm_l$-reprÈsentations de $GL_{sg}(K)$, on a l'ÈgalitÈ suivante:
$$r_l \Bigl ( \st_s(\pi) \Bigr )=\sum_{\underline i \in \IC_\varrho(s)} \st_{\underline i}(\varrho).$$
Par ailleurs pour tout $\underline i \in \IC_\varrho(s)$ et pour tout parabolique $P$, $J_{P} \Bigl (
\st_{\underline i} (\varrho) \Bigr )$ est Ègal ‡ la somme des constituants irrÈductibles de $\varrho$-niveau $\underline i$ de $r_l \Bigl (J_{P}(\st_s(\pi)) \Bigr )$.
\end{theo}

\rem pour $s<m(\varrho)$,
la rÈduction modulo $l$ de $\st_{s}(\pi)$ est irrÈductible.

\begin{defi}
On dira que  $l$ est \textit{banal} pour $GL_d(K)$ si $e_l(q)>d$.
\end{defi}

\rem dans le cas banal toute reprÈsentation cuspidale est supercuspidale, i.e. $m(\varrho) < s$ avec les notations
prÈcÈdentes.

\subsection{de \texorpdfstring{$D_{K,d}^\times$}{Lg} ‡ coefficients dans \texorpdfstring{$\overline \Fm_p$}{Lg} et leurs relËvements}
\label{para-repD}

Soit $\tau$ une $\overline \Qm_l$-reprÈsentation irrÈductible de $D_{K,d}^\times$ que
l'on suppose $l$-entiËre, i.e. de caractËre central $l$-entier. Quitte ‡ tordre ce
caractËre central par un caractËre non ramifiÈ, on le suppose trivial sur l'uniformisante
$\varpi$ et donc $\tau$ est une reprÈsentation de $D_{K,d}^\times/\varpi^\Zm$.
On note $\PC_{K,d}$ le radical de $\DC_{K,d}$ et soit
$$1+\PC_{K,d} \subset \DC_{K,d}^\times \subset D_{K,d}^\times/\varpi^\Zm$$
la filtration de quotients successifs $\Fm_{q^d}^\times$ et $\Zm/d\Zm$.

On choisit alors un facteur irrÈductible $\zeta$ de $\tau_{|1+\PC_{K,d}}$; on note
$N_\zeta$ le normalisateur de sa classe d'isomorphisme dans 
$D_{K,d}^\times/\varpi^\Zm$ et soit $\tilde \zeta$ son prolongement ‡ $N_\zeta$:
en effet $1+\PC_{K,d}$ Ètant un pro-$p$-groupe, la dimension de $\zeta$
est une puissance de $p$ de sorte que, un $p$-Sylow de $N_\zeta/(1+\PC_{K,d})$
Ètant cyclique, $\zeta$ admet un prolongement ‡ $N_\zeta$, cf. 
\cite{vigneras-apropos} lemme 1.19.

\begin{prop} (cf. \cite{dat-jl} proposition 2.3.2) \label{prop-repD} \\
Il existe un caractËre $\chi$ tel que
$$\tau \simeq \ind_{J}^{D_{K,d}^\times/\varpi^\Zm} \bigl ( \tilde \zeta_{|J} \otimes \chi \bigr ),$$
o˘ $J$ est un sous-groupe de $D_{K,d}^\times/\varpi^\Zm$ contenant 
$1+\PC_{K,d}$ la forme $N_\zeta \cap N_\chi$,
pour $N_\chi$ est le normalisateur de $\chi$ vÈrifiant les points suivants:
\begin{itemize}
\item $J$ contient $N_\zeta \cap \DC_{K,d}^\times$;

\item il existe des entiers $f',d',e'$ de produit Ègal ‡ $d$ tels que
$$J/(1+\PC_{K,d}) \simeq \Fm^\times_{q^{f'd'}} \rtimes m \Zm/e'd' \Zm,$$
o˘ le gÈnÈrateur de $\Zm/e'd'\Zm$ agit par le Frobenius relatif ‡ $\Fm_{q^{f'}}$
et $m$ est un diviseur de $d'$ tel que 
$$f'm=[D_{K,d}^\times/\varpi^\Zm:\DC_{K,d}^\times J]$$
est le cardinal $e_\tau$ de la classe d'Èquivalence inertielle de $\tau$.

\item L'abÈlianisÈ de $J/(1+\PC_{K,d})$ s'identifie via le morphisme norme,
‡ $\Fm_{q^{f'm}}^\times \times m \Zm/e'd' \Zm$.

\item La rÈduction modulo $l$ de $\tau$ est de cardinal
$$r_\tau=[N_{\zeta} \cap N_{r_l(\chi)}:J]$$
et de la forme 
$$[\bar \tau] + [\bar \tau \nu] + \cdots + [\bar \tau \nu^{r-1}]$$
pour $\bar \tau$ une $\overline \Fm_l$-reprÈsentation irrÈductible de $D_{K,d}^\times$
et $\nu$ le caractËre $g \mapsto q^{\val \circ \nrd (g)}$.

\item Une reprÈsentation $\tau'$ a mÍme rÈduction modulo $l$ si et seulement
si $J'=J$ et $r_l(\chi'), r_l(\chi)$ sont conjuguÈs sous $N_\zeta \cap N_{r_l(\chi)}$.

\item Le nombre $n(\tau)$ de reprÈsentations irrÈductibles strictement congrues ‡ $\tau$
est donnÈe par la formule $\frac{n(\chi)}{[N_{\zeta} \cap N_{r_l(\chi)}:J]}$
o˘ $n(\chi)$ est Ègal ‡ la plus grande puissance de $l$ qui divise
$a(\tau):=\frac{d}{e_\tau} (q^{e_\tau}-1)$.
\end{itemize}
\end{prop}

\rem le diviseur $e'$ de $d$ correspond au diviseur $s$ de $d$ tel que
$\tau$ est de la forme $\pi[s]_D$, i.e. correspond par la correspondance
de Jacquet-Langlands ‡ la reprÈsentation de Steinberg gÈnÈralisÈe $\st_s(\pi)$
o˘ $\pi$ est une reprÈsentation irrÈductible cuspidale de $GL_g(K)$ avec $d=sg$.

\begin{notas} \label{nota-D} 
Suivant la proposition prÈcÈdente, on notera 
\begin{itemize}
\item $m(\bar \tau)=[N_\chi \cap N_{r_l(\chi)}:J]$,

\item $s(\bar \tau)$ la plus grande puissance de $l$ divisant $\frac{d}{m(\bar \tau)g(\bar \tau)}$,

\item $g_{-1}(\bar \tau):=g(\bar \tau):=\frac{d}{e'}=f'd'$ et

\item pour $0 \leq i \leq s(\bar \tau)$,  $g_i(\bar \tau)=m(\bar \tau)l^i g(\bar \tau)$.
\end{itemize}
\end{notas}

\rem pour $r_l(\tau)$, rappelons que comme $r_l(\tilde \zeta_{|J})$ est irrÈductible,
$r_l(\tau)$ est de longueur $[N_\zeta \cap N_{r_l(\chi)}:N_\zeta \cap N_\chi]$.
En particulier $r_l(\tau)$ est irrÈductible si et seulement si $N_\chi=N_{r_l(\chi)}$.

\begin{defi}
Soit $\bar \tau$ une $\overline \Fm_l$-reprÈsentation irrÈductible de $D_{K,d}^\times$ dont
le caractËre central est trivial sur $\varpi^\Zm \subset K^\times$.
On  dÈfinit 
$$\CC_{\bar \tau} \subset \rep_{\Zm_l^{nr}}^\oo(D_{K,d}^\times)$$
la sous-catÈgorie pleine formÈe des $\Zm_l^{nr}$-reprÈsentations de $D_{K,d}^\times$ dont tous les 
$\Zm^{nr}\DC_{K,d}^\times$-sous-quotients irrÈductibles sont isomorphes ‡
un sous-quotient de $\bar \tau_{|\DC_{K,d}^\times}$.
\end{defi}

\begin{prop} (cf. \cite{dat-torsion} \S B.2) \label{prop-scindage} \\
Soit $P_{\bar \tau^0}$ une enveloppe projective de $\bar \tau^0$ dans
$\rep^\oo_{\Zm_l^{nr}}(\DC_{K,d}^\times)$. Alors la sous-catÈgorie 
$\CC_{\bar \tau}$ est facteur direct dans $\rep_{\Zm_l^{nr}}^\oo(D_{K,d}^\times)$
pro-engendrÈe par l'induite $P_{\bar \tau}:=\ind_{\DC_{K,d}^\times}^{D_{K,d}^\times} (P_{\bar \tau^0})$.
\end{prop}

\begin{nota} \label{nota-RC} Pour tout $d \geq 1$, on note
$\RC_{ \overline \Fm_l}(d)$ l'ensemble des classes d'Èquivalence des 
$\overline \Fm_l$-reprÈsentations irrÈductibles de $D_{v,d}^\times$ dont
le caractËre central est trivial sur $\varpi^\Zm \subset K^\times$.
\end{nota}

Ainsi toute $\Zm_l^{nr}$-reprÈsentation $V_{\Zm_l^{nr}}$ de $D_{K,d}^\times$
se dÈcompose en une somme directe
\begin{equation} \label{eq-tau-iso}
V_{\Zm_l^{nr}} \simeq \bigoplus_{\bar \tau \in \RC_{\overline \Fm_l}(d)} V_{\Zm_{l,\bar \tau}^{nr}}
\end{equation}
o˘ $V_{\Zm_{l,\bar \tau}^{nr}}$
est un objet de $\CC_{\bar \tau}$, i.e. tous ses sous-quotients irrÈductibles
sont isomorphes ‡ un sous-quotient de $\bar \tau_{|\DC_{K,d}^\times}$.

\begin{theo} (cf. \cite{dat-jl} 3.1.4) \\ \label{theo-jl}
Il existe une bijection
$$\Bigl \{ \overline \Fm_{l}-\hbox{reprÈsentations superspeh de }GL_{d}(K) \Bigr \}  \simeq 
\Bigl \{ \overline \Fm_{l}-\hbox{reprÈsentations irrÈductibles de } D^{\times}_{K,d} \Bigr \} $$
compatible ‡ la rÈduction modulo $l$ au sens suivant:
\begin{itemize}
\item soit $\varrho$ une $\overline \Fm_{l}$-reprÈsentation irrÈductible supercuspidale de $GL_{g}(K)$ avec $d=sg$;

\item soit $\pi$ un $\overline \Qm_{l}$-relËvement de $\varrho$.
\end{itemize}
Alors la rÈduction modulo $l$ de $\pi[s]_{D}$, notÈe $\bar \tau_{\varrho,t}$, 
est irrÈductible et correspond via la bijection ci-dessus ‡ la superspeh $\speh_{s}(\varrho)$.
\end{theo}

\rem avec les notations prÈcÈdentes, $g(\bar \tau)$ est Ègal au $g$ du $\varrho$ dans le thÈorËme
ci-avant o˘ $\varrho \in \scusp_{-1}(\bar \tau)$.

\begin{defi} \label{defi-tau-type}
On dira d'une $\overline \Fm_l$-reprÈsentation de $D^\times_{K,d}$ (resp. de $\DC_{K,d}^\times$) qu'elle
est de type $\varrho$ si tous ses constituants irrÈductibles sont, via la bijection prÈcÈdente, 
image d'une superSpeh $\speh_s(\varrho \otimes \chi \circ \det)$ o˘ $\chi$ est un caractËre non ramifiÈ
de $K^\times$. 
\end{defi}

\begin{nota}
On notera alors
$\RC_{\overline \Fm_l}(h,\varrho)$ le sous-ensemble de $\RC_{\overline \Fm_l}(h)$
constituÈ des $\bar \tau$ de type $\varrho$ ainsi que
$$\RC_{\overline \Fm_l}(\varrho)=\coprod_{h=tg(\varrho)} \RC_{\overline \Fm_l}(h,\varrho).$$
\end{nota}

\rem si $h$ n'est pas divisible par $g(\varrho)$ alors $\RC_{\overline \Fm_l}(\varrho)$ est vide.

\begin{nota} \label{nota-scusp}
Pour $1 \leq g$, on notera $\scusp_K(g)$ l'ensemble des classes d'Èquivalences inertielles
des $\overline \Fm_l$-reprÈsentation irrÈductibles supercuspidales de $GL_g(K)$.
Pour $\bar \tau \in \RC_{\overline \Fm_l}(g(\varrho),\varrho)$, on notera aussi
$$\scusp_K(\bar \tau)=\scusp_K(\varrho).$$
\end{nota}

\rem on notera aussi que $\bar \tau \in \RC_{\overline \Fm_l}(h)$ possËde exactement un type, i.e.
\begin{equation} \label{eq-tau-type}
\RC_{\overline \Fm_l}(h)=\coprod_{g | h} ~
\coprod_{\varrho \in \scusp_K(g)} \RC_{\overline \Fm_l}(h,\varrho).
\end{equation}

\noindent \textit{Exemple:} 
soit $\varrho$ une $\overline \Fm_l$-reprÈsentation irrÈductible supercuspidale de $GL_{g(\varrho)}(F_v)$. 
Pour tout $i \geq -1$, on considËre $\pi_i$ un relËvement de $\varrho_i$. Soit alors $t \geq 1$ et 
$\bar \tau$ la rÈduction modulo $l$ de $\pi_{-1}[t]_D$ laquelle est irrÈductible. Pour $i \geq 0$ et $t_i$ 
tel que  $t_ig_i(\varrho)=tg(\varrho)$, la reprÈsentation $\pi_i[t_i]_D$ 
(resp. tout sous-quotient irrÈductible de la rÈduction modulo $l$ de $\pi_i[t_i]_D$) 
appartient ‡ $\CC_{\bar \tau}$. 
RÈciproquement pour tout $\tau' \in \CC_{\bar \tau}$ une $\overline \Qm_l$-reprÈsentation
irrÈductible entiËre, il existe $i \geq -1$ et une reprÈsentation irrÈductible cuspidale
$\pi_i$ de $GL_{g_i(\varrho)}(K)$ dont le support supercuspidal de sa rÈduction modulo $l$
est un segment de Zelevinsky-VignÈras de longueur $m(\varrho)l^i$ et telle que
$\tau' \simeq \pi_i[\frac{d}{g_i(\varrho)}]_D$.

\begin{defi} \label{defi-tautype}
Suivant la discussion prÈcÈdente, on dira d'une $\overline \Qm_l$-reprÈsentation irrÈductible 
$\tau' \in \CC_{\bar \tau}$ qu'elle est de $\bar \tau$-type $i$.
On notera aussi $\scusp_i(\bar \tau)$ l'ensemble des classes d'Èquivalence de ces reprÈsentations
$\pi_i$ et $\scusp(\bar \tau)=\bigcup_{i \geq -1} \scusp_i(\bar \tau)$.
\end{defi}

\bibliographystyle{plain}
\bibliography{bib-ok}

\end{document}